\documentclass[review]{article}

\usepackage{lipsum}
\usepackage{amsfonts}
\usepackage{amsmath}
\usepackage{amsthm}
\usepackage{graphicx}
\usepackage{epstopdf}
\date{}

\usepackage{geometry}
\usepackage[]{algorithm}
\usepackage{algorithmicx}
\usepackage{algpseudocode}
\newtheorem{theorem}{Theorem}[subsection]
\newtheorem{assumption}{Assumption}[subsection]
\ifpdf%
  \DeclareGraphicsExtensions{.eps,.pdf,.png,.jpg}
\else
  \DeclareGraphicsExtensions{.eps}
\fi
\usepackage{amsopn}

\usepackage{booktabs}

\newcommand{\TheTitle}{%
  Advances in Implementation, Theoretical Motivation, and Numerical Results for the Nested Iteration with Range Decomposition Algorithm
}

\newcommand{\TheName}{%
  Wayne Mitchell
}

\newcommand{\TheAddress}{%
  University of Colorado at Boulder, 526 UCB, Boulder CO 80309-0626
  (wayne.mitchell@colorado.edu)
}

\newcommand{\TheFunding}{This work was  performed under the auspices of the U.S. Department of Energy under 
grant numbers (SC) DE-FC02-03ER25574 and (NNSA) DE-NA0002376, 
Lawrence Livermore National Laboratory under contract B614452.
}

\newcommand{\TheCollaborators}{%
  Tom Manteuffel
}

\author{\TheName\thanks{\TheAddress}}
\title{{\TheTitle}\thanks{\TheFunding}}
\ifpdf%
\fi

\DeclareMathOperator*{\argmin}{arg\,min}
\begin{document}
\DeclareRobustCommand{\Chi}{\raisebox{2pt}{$\chi$}}

\maketitle

\begin{center}
In collaboration with:
  {\TheCollaborators}
\end{center}
\vspace{1cm}

\begin{abstract}
This paper studies a low-communication algorithm for solving elliptic partial differential equations (PDE's) on high-performance machines, the nested iteration with range decomposition algorithm (NIRD). Previous work by Appelhans et al. in \cite{Appelhans} has shown that NIRD converges to a high level of accuracy within a small, fixed number of iterations (usually one or two) when applied to simple elliptic problems. 
This paper makes some improvements to the NIRD algorithm (including the addition of adaptivity during preprocessing, wider choice of partitioning functions, and modified error measurement) that enhance the method's accuracy and scalability, especially on more difficult problems. In addition, an updated convergence proof is presented based on heuristic assumptions that are supported by numerical evidence.
Furthermore, a new performance model is developed that shows increased performance benefits for NIRD when problems are more expensive to solve using traditional methods. Finally, extensive testing on a variety of elliptic problems provides additional insight into the behavior of NIRD and additional evidence that NIRD achieves excellent convergence on a wide class of elliptic PDE's and, as such, should be a very competitive method for solving PDE's on large parallel computers.
\end{abstract}




\section{Introduction}

When solving elliptic partial differential equations (PDE's), multigrid algorithms often provide optimal solvers and preconditioners capable of providing solutions with $O(N)$ computational cost, where $N$ is the number of unknowns. As parallelism of modern super computers continues to grow towards exascale, however, the cost of communication, especially latency costs, has overshadowed the cost of computation as the next major bottleneck for multigrid algorithms \cite{Gahvari:2011dh}. Typically, multigrid algorithms require $O((\log P)^2)$ communications between processors in order to solve a PDE problem to the level of discretization accuracy, where $P$ is the number of processors. Efforts to reduce the cost of communication for multigrid algorithms is a major area of research, and various approaches to this problem are currently being explored. A good survey of much of this work is given in \cite{Chow:2010ti}.

One approach to reducing communication in multigrid algorithms is to optimize their parallel implementation. These approaches do not change the underlying algorithms, but rather seek to achieve speedups through optimizing operations such as sparse matrix-vector multiplications \cite{DBLP:journals/corr/BienzGO16} or determining processor neighbors \cite{Anonymous:ycC62NK_}. As such, the original convergence properties of the underlying methods are retained, but reductions in communication cost resulting from these optimizations are relatively modest.

A different approach is to modify the multigrid algorithm itself. Much work has been done in this area, especially involving efforts to reduce complexity of operators and density of communication patterns on the coarser grids of multigrid hierarchies \cite{DeSterck:2006et,DeSterck:2008fc}. Although such modifications of the underlying multigrid hierarchies can be detrimental to convergence, the savings in communication cost is often more than enough to ensure a faster method overall. 

A third approach to the problem of high communication costs is to design new algorithms altogether, starting from the premise of minimizing communication. Some notable examples of this kind of work include the algorithms developed by Mitchell \cite{Mitchell:2004hz, Mitchell:2016vg, Mitchell:1998kr} and Bank \cite{Bank:1999uq, Bank:2006hn}. This paper studies the recently developed, low-communication algorithm, Nested Iteration with Range Decomposition (NIRD) \cite{Appelhans}, which aims to solve elliptic PDE's with $O(\log P)$ communications by trading some additional computation on each processor for a significant reduction in communication cost. 

Traditional parallel PDE solvers split the work of solving a problem among processors by partitioning the computational domain or degrees of freedom \cite{Chow:2010ti}. Each processor then owns information about the problem and is responsible for the solution of the problem over its assigned partition of the domain, which is referred to here as that processor's home domain. Problems that are parallelized in this way require inter-processor communication for many important global operations, such as the communication of boundary information with nearest processor neighbors to perform matrix-vector multiplications. Thus, these communications must occur many times over the course of a single multigrid cycle. To avoid this frequent need for communication, NIRD decomposes the problem itself rather than the computational domain. The resulting subproblems are defined over the global computational domain and are able to be stored and solved on individual processors without the need for communication. Solutions to these subproblems are then recombined (involving communication) in order to improve a global solution to the original problem. 

Previous work on NIRD by Appelhans et al. \cite{Appelhans} developed an initial implementation of the algorithm and tested it on model Poisson and advection-diffusion problems. Ultimately, it was shown that, for the model Poisson problem, NIRD achieves accuracy similar to traditional methods within a single iteration with $\log_2 P$ communication steps, and some convergence theory was established to motivate this single-cycle convergence. This paper extends the work done in \cite{Appelhans} by reimplementing NIRD with some improvements to the implementation over the previous code, updating and examining convergence theory, and applying NIRD to a wider set of test problems to gain a better overall understanding of the algorithm's behavior, the convergence theory, and the class of problems for which NIRD performs well.

Section \ref{preliminaries:sec} begins by introducing some preliminary information and tools used throughout the rest of the paper and is followed by a description of NIRD itself in Section \ref{NIRD_description:sec}. Section \ref{NIRD_details:sec} provides a more detailed discussion of some of the improvements to the algorithm over the previous implementation used in \cite{Appelhans} and shows brief numerical examples demonstrating the importance of these improvements. In Section \ref{NIRD_convergence_theory:sec}, the examination of numerical results leads to the identification and verification of realistic, heuristic assumptions that form the basis of an updated and generalized convergence proof. A performance model is then developed in Section \ref{NIRD_performance_model:sec}, which demonstrates the insight that NIRD not only achieves significant reduction in the number of communications compared with a traditional nested iteration approach but also provides even larger savings for some problems that are more difficult and expensive to solve using traditional methods. Finally, Section \ref{NIRD_numerical_tests:sec} shows numerical results for NIRD applied to a variety of elliptic test problems. These results improve our overall understanding of how the algorithm behaves and provide numerical evidence that supports the heuristics of the convergence theory shown in Section \ref{NIRD_convergence_theory:sec}. In addition, NIRD is shown to perform well for elliptic problems with a variety of difficulties as well as for problems using higher-order finite elements.


\section{Preliminaries}
\label{preliminaries:sec}

This section briefly describes the discretization method and adaptive mesh refinement strategy used throughout this paper. The traditional nested iteration approach is also described. 


\subsection{First-order system least-squares (FOSLS)}
\label{FOSLS:sec}

Though the NIRD algorithm may be applied to any discretization that yields a locally sharp a posteriori error estimate, the discussion of NIRD in this paper specifically uses a first-order system least-squares (FOSLS) method to discretize the PDE to be solved. The FOSLS method firsts casts the PDE as a first-order system, $L\mathbf{U} = \mathbf{f}$, where $L$ is the first-order differential operator associated with the PDE, and $L$ is assumed to be linear. Then a solution over some Hilbert space, $\mathcal{W}$, is obtained by minimizing 
\begin{align}
\mathbf{U} =\arg \min_{\mathbf{V}\in \mathcal{W}} || L\mathbf{V} - \mathbf{f} ||,
\end{align}
where $|| \cdot ||$ is the $L^2$ norm. The least-squares functional (LSF), $|| L\mathbf{V} - \mathbf{f} ||$, is continuous and coercive in the $H^1$ Sobolev norm for general second-order, uniformly elliptic PDE's in 2D and 3D, as shown by Cai et al. in \cite{Cai,Cai:1997bq}. This implies optimal error estimates for finite element approximation by subspaces of $H^1$ and the existence of optimal multigrid solvers for the resulting discrete problems. This also allows the local least-squares functional (LLSF), $||L\mathbf{U} - \mathbf{f}||_\tau$, evaluated over a finite element, $\tau$, to be used as an a posteriori error estimate for adaptive mesh refinement, which is an important component of the NIRD algorithm. 


\subsection{Accuracy-per-computational-cost efficiency-based refinement (ACE)}
\label{ACE:sec}

The adaptive mesh refinement (AMR) strategy used throughout this paper is designed to maximize ``accuracy-per-computational-cost" efficiency (ACE). The ACE refinement method, originally proposed by DeSterck et al. in \cite{DeSterck:2008js}, aims to refine a fraction, $r\in(0,1]$, of the total number of elements at a given refinement level by minimizing an effective functional reduction factor, $\gamma(r)^{1/W(r)}$, where $\gamma(r)$ is the functional reduction expected after refining $r$ of the elements and $W(r)$ is the anticipated amount of work required to solve on that refined grid. This is equivalent to choosing
\begin{align}
r_{opt} = \arg \min_{r\in (0,1]}\frac{\log(\gamma(r))}{W(r)}.
\end{align}
The logic behind this minimization is the desire to maximize functional reduction per unit of work. ACE refinement has been shown to perform well on a wide variety of problems, including FOSLS discretizations of PDE's where the FOSLS functional is used as the a posteriori error measure for refinement \cite{Adler:2011hp}.


\subsection{Nested iteration}
\label{intro_NI:sec}

Nested iteration is an efficient, multilevel method for solving PDE's that proceeds by solving first on a very coarse grid, refining, and interpolating the coarse-grid solution to serve as an initial guess on the next grid. This process of solving and interpolating is then repeated until a desired finest grid is reached. Since it leverages coarse-grid solutions, which are relatively cheap to obtain, nested iteration is usually much more efficient than simply solving the fine grid problem by starting with a random or zero initial guess. Furthermore, when paired with a suitable adaptive refinement scheme (such as ACE), nested iteration not only provides a solution with optimal cost but also generates an optimal fine-grid mesh for that solution.

In a geometric multigrid setting, where the solves on each grid level are accomplished by V-cycles, nested iteration is sometimes also referred to as full multigrid (FMG). In this case, it is straightforward to show that a single FMG cycle (using a sufficient, constant number of V-cycles to solve on each grid level) yields a solution with accuracy on the order of the discretization on the finest grid with $O(N)$ computational cost, where $N$ is the number of fine-grid unknowns. A brief inductive proof of this fact can be found in \cite{McCormick:2016vp}.

Note that, while nested iteration (or FMG) has optimal scaling in terms of computational cost, it is not necessarily optimal in terms of communication cost. The efficiency of FMG comes from its ability to leverage information from coarser grids, which are computationally much cheaper than the fine grid. This strength does not translate to communication cost for traditional nested iteration implementations, however, due to the need for a constant number of communications on each level of the multigrid hierarchy. Thus, traditional nested iteration in a weak scaling context (where the number of unknowns per processor is kept fixed, and the number of processors is scaled) accrues an $O((\log P)^2)$ communication cost. Similarly, applying standard V-cycles to the fine-grid problems yields $O((\log P)^2)$ communication cost (though with a larger constant than FMG) due to the need for $O(\log P)$ cycles to solve the problem. Again, the NIRD algorithm attempts to significantly reduce this communication cost to $O(\log P)$. 


\section{NIRD algorithm description}
\label{NIRD_description:sec}

The NIRD algorithm begins with a preprocessing step performed simultaneously and redundantly on all processors, requiring no communication. First, solve the original PDE until there are at least $P$ elements in the mesh, where $P$ is the number of processors. Note that, in this paper, the preprocessing step uses adaptive mesh refinement (discussed further in Section \ref{NIRD_preprocessing:sec}) where previous implementation in \cite{Appelhans} used uniform refinement. This coarse mesh is then partitioned into home domains, $\Omega_l$, for each processor, $l$, and a corresponding partition of unity (PoU) is defined. As a simple example, consider the PoU defined by the characteristic functions,
\begin{align}
\Chi_l = \begin{cases}
1, \,\, \text{in } \Omega_l, \\
0, \,\, \text{elsewhere} .
\end{cases}
\label{disctsChi}
\end{align}
Note that adaptive refinement should place approximately equal portions of the error in each $\Omega_l$, though the $\Omega_l$'s may be very different sizes. Also note that what is meant by ``error" here is left intentionally vague, since the NIRD algorithm may be applied to any discretization method with an accurate a posteriori error measure. The discussion here uses a FOSLS discretization along with ACE refinement as described in Sections \ref{FOSLS:sec} and \ref{ACE:sec}. 

The definition of the home domains and corresponding PoU completes the preprocessing step. As suggested by its name, the NIRD algorithm then decomposes the PDE problem by decomposing the range (or right-hand side) of the residual equation:
\begin{align*}
L\delta\mathbf{u} = \mathbf{f} - L\mathbf{u}_0^h = \sum_l \Chi_l (\mathbf{f} - L\mathbf{u}_0^h),
\end{align*}
where $\mathbf{u}^h_0$ is the coarsely meshed preprocessing solution. Each processor then adaptively solves a subproblem over the entire domain for its right-hand side:
\begin{align}
\delta \mathbf{u}_l^h = \arg\min\limits_{\mathbf{v}^h \in W_l} ||L\mathbf{v}^h - \mathbf{f}_l||,
\label{subproblem_min:eq}
\end{align}
where $\mathbf{f}_l = \Chi_l (\mathbf{f} - L\mathbf{u}_0^h)$ is non-zero only on or near processor $l$'s home domain.

This solve is the main computation step of the algorithm and involves no communication, since each processor solves its subproblem over the entire domain. Note that the PoU should be chosen so that the support of the right-hand side, $\mathbf{f}_l$, is localized on or near the home domain, $\Omega_l$, for each processor. This should lead to an adaptive mesh that is refined primarily in, and remains coarse away from, the home domain. This allows each processor to obtain accurate subproblem solutions while using a reasonable number of elements. The global solution update is then calculated by summing the solutions from each processor:
\begin{align*}
\mathbf{u}_1^h = \mathbf{u}_0^h + \sum_l \delta \mathbf{u}_l^h.
\end{align*}
This sum can be accomplished in exactly $\log_2 P$ communication steps in which processors communicate in a predetermined, pairwise fasion \cite{Appelhans} and yields a solution that lives on a globally refined ``union mesh" obtained by taking the union of the meshes from each subproblem. A new global residual may then be formed and new updates calculated to produce an iteration. Algorithm \ref{nird:alg} shows the pseudo code for the NIRD algorithm in its entirety.

\begin{algorithm}[H]
\caption{NIRD Algorithm.}
\label{nird:alg}
\begin{algorithmic}
	\State Solve $L\mathbf{u}_0^h = \mathbf{f}$ using NI with AMR until the coarse mesh has at least $P$ elements.
	\State Partition the coarse mesh into home domains, $\Omega_l$.
	\State Define a PoU with characteristic functions, $\Chi_l$.
	\For {$i = 0 \rightarrow$ (num iterations)}
		\State Let $\mathbf{f}_l =  \Chi_l (\mathbf{f} - L\mathbf{u}^h_{i})$, where $l$ is this processor's rank.
		\State Solve the subproblem, $L\delta \mathbf{u}^h_l =\mathbf{f}_l$, using NI with AMR.
		\State Update the global solution, $\mathbf{u}^h_{i+1} = \mathbf{u}^h_{i} + \sum_l \delta \mathbf{u}^h_l$ (communication step).
	\EndFor
\end{algorithmic}
\end{algorithm}

Having described the NIRD algorithm, it is worth comparing it here to other low-communication algorithms that have been previously proposed. Perhaps the most natural comparison is with the parallel adaptive meshing paradigm presented by Bank and Holst in \cite{Bank:1999uq, Bank:2006hn}, which shares several similarities with NIRD. After an adaptive, preprocessing step, in which a global coarse mesh is partitioned into home domains with approximately equal error, each processor independently does its own adaptive refinement focusing on its home domain. These independent, adaptively refined meshes are then unioned to form the global mesh, very much like what is done in NIRD. The main distinction between the algorithms is how the actual global solution on the union mesh is obtained. The Bank-Holst paradigm does not partition the right-hand side. Rather each processor solves their subproblem with the full right-hand side, and the adaptive refinement is guided to focus in the home domain by simply multiplying the a posteriori error measure by $10^{-6}$ outside the home domain. The subproblem solutions are then used as initial guesses over the home domains, and the final global solution is obtained through standard, parallel domain decomposition or multigrid techniques. Thus, while this algorithm generates the adaptively refined parallel mesh very efficiently, obtaining the solution on that mesh still requires the application of less communication-efficient, traditional methods. NIRD, on the other hand, achieves both a good global mesh and global solution with very low communication cost. By partitioning the right-hand side and allowing adaptive refinement on each processor to proceed as usual, NIRD generates subproblem solutions that are well-resolved and thus yield a good global solution when added together. 

Another important note about both the Bank-Holst parallel adaptive meshing algorithm and NIRD is their ease of implementation. Both algorithms very naturally leverage existing software and solution techniques in order to perform the serial, adaptive solves on each processor. The only new pieces of code that need to be written are the communication routines and implementation of the partition of unity characteristic functions for NIRD. This makes NIRD relatively easy to implement within the framework of an existing finite element or solver package.


\section{Algorithm details}
\label{NIRD_details:sec}

This section gives some further details on a few key pieces of the NIRD algorithm that are newly implemented or improved since the previous implementation used in \cite{Appelhans}. First, the details of the adaptive preprocessing step are shown along with a brief numerical example showing the importance of this step to the overall accuracy of NIRD. This is followed by some significant additions to the choice of characteristic functions for the partition of unity. The final subsection reveals that naive functional measurement on the NIRD subproblems is potentially problematic as an error measure for adaptive refinement and proposes a fix to restore the validity of the functional as an a posteriori error measure.


\subsection{The preprocessing step}
\label{NIRD_preprocessing:sec}

One contribution of this work is the implementation of an adaptive preprocessing step. Although the preprocessing step was discussed in \cite{Appelhans}, that implementation was not adaptive. Instead, the initial coarse mesh was uniformly refined until it had at least $P$ elements, which may result in unequal distribution of the error across processors and hinder overall NIRD performance. Since computational effort and resources are equally distributed across processors, the error should be as well. Using adaptive refinement to establish the coarse mesh should roughly equally distribute error among processors. Allowing for more elements in the coarse mesh gives adaptive refinement the best chance of equally distributing error among those elements (and subsequently among processors), but processors should still reserve most of their memory resources for solving the local subproblem. This raises the question of how much refinement should be done during the preprocessing step. 

For the purposes of this paper, the preprocessing step enforces $N_c \geq P$, where $N_c$ is the number of elements in the coarse mesh. This is only for ease of implementation and discussion: as $P$ gets  large compared to the memory available on each processor, this is clearly untenable. In this case, one may consider $P$ to be the number of nodes or clusters of processors and apply NIRD at the level of parallelism between these nodes. Such an approach would still result in greatly enhanced performance, as inter-node communication is generally much more expensive than on-node communication.

With the above consideration in mind, the strategy for stopping refinement during the preprocessing step may then be described as follows. If $E$ is the total number of elements that may be stored on a processor, most of these, say 90\%, should be reserved for solving the processor's subproblem. If it is possible to equidistribute the error with fewer elements, however, it is preferable to use as few elements as possible. Judging whether error is well-equidistributed can be tricky and requires the major features (sources of error) of the problem to be resolved with $0.1*E$ or fewer elements. Assuming this is the case, the preprocessing step considers error to be equidistributed if the ratio, $\eta_1/\eta_0$, is below some threshold factor, say 10, where $\eta_1$ is the largest error over a home domain and $\eta_0$ is the smallest error over a home domain. Again, for the purposes of this paper, local error is measured via the local least-squares functional. To summarize, the preprocessing step refines while $(N_c < P)$ or $((N_c < 0.1*E)$ and $(\eta_1/\eta_0 > 10))$.
%

To demonstrate the importance of the adaptive preprocessing step, consider a 2D Poisson problem with the given right-hand side:
\begin{align*}
-\Delta p &= f, \, \text{ on }  \Omega = [0,1]\times[0,1], \\
p &= 0 , \, \text{ on } \partial\Omega, \\
f(x,y) &= \begin{cases} 100, \, (x,y)\in[0.49,0.51]\times[0.49,0.51] \\ 0, \text{ else} \end{cases}.
\end{align*}
The very localized, discontinuous right-hand side demands adaptive refinement during the preprocessing step in order to equidistribute error. Comparing the overall performance of NIRD with and without an adaptive preprocessing step on this problem shows a significant increase in accuracy when using adaptivity as shown in Figure \ref{preprocessing_poisson}. The LSF values shown in this plot are normalized by the LSF of the preprocessing solution using uniform refinement. The reference lines shown are the expected accuracies using a traditional nested iteration approach using $N_U$ or $N_T$ elements, where $N_U$ is the number of elements in the union mesh produced by NIRD and $N_T>N_U$ is the total number of elements used by NIRD for all subproblems. Note that the union meshes produced by NIRD with and without the adaptive preprocessing step have similar numbers of elements, but their accuracies are quite different. This indicates that without the adaptive preprocessing step, NIRD is unable to generate a union mesh with optimal element distribution. As such, the accuracy suffers significantly and will continue suffer more as $P$ increases, whereas with the adaptive preprocessing step, the NIRD iterates quickly achieve accuracy very similar to that achieved by a traditional nested iteration approach using ACE refinement. Thus, the adaptive preprocessing step has allowed NIRD to generate a nearly optimal union mesh and restored scalability for NIRD, even for this problem with a highly localized, discontinuous right-hand side.

\begin{figure}[t]
\centering
\includegraphics[width=0.65\textwidth]{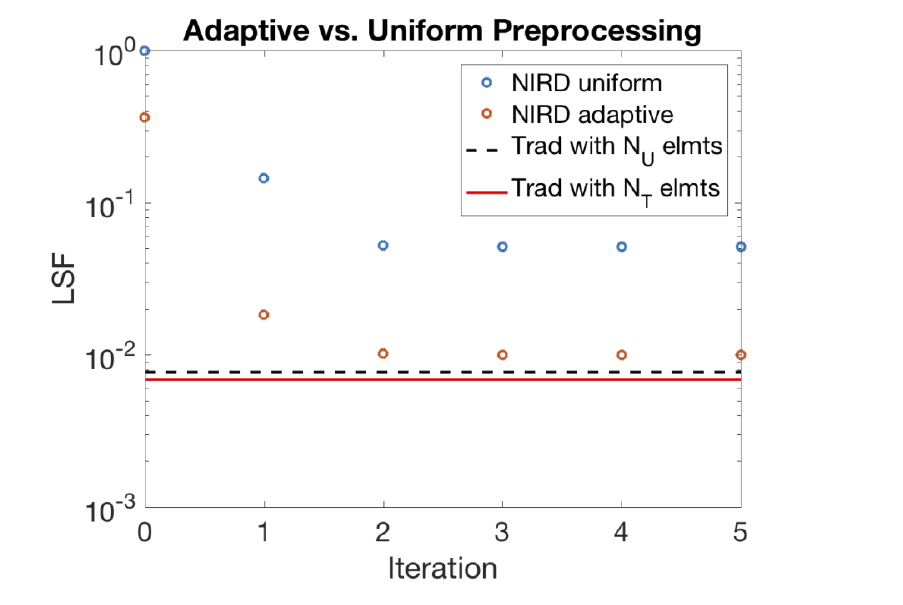}
\caption{Comparison of NIRD convergence with and without an adaptive preprocessing step using 1,024 processors.}
\label{preprocessing_poisson}
\end{figure}


\subsection{Partitions of unity}
\label{NIRD_pou:sec}

In addition to an adaptive preprocessing step, another advancement in the implementation of the NIRD algorithm is the ability to construct continuous or even infinitely smooth characteristic functions for the partition of unity (PoU). Previous implementation of the algorithm supported only the discontinuous PoU shown in equation (\ref{disctsChi}). As mentioned in Section \ref{NIRD_description:sec}, PoU's should be chosen so that the characteristic functions have local support near a processor's home domain, thereby inducing adaptive refinement that focuses on the home domain as much as possible. The discontinuous characteristic functions in equation (\ref{disctsChi}) have support exactly over $\Omega_l$, but these may not be the only good choice for a PoU. A $C^0$ PoU may be constructed from piecewise-linear nodal basis functions over the coarse elements in the preprocessing mesh with nodal values 
\begin{align*}
\Chi_l(x_i) = \begin{cases}
1/n_i, \,\, x_i\in \Omega_l \\
0, \,\, x_i \notin \Omega_l
\end{cases},
\end{align*}
where $n_i$ is the number of home domains incident to node $x_i$. This yields $C^0$ characteristic functions but also expands their support to neighboring elements in the coarse, preprocessing mesh, resulting in more refinement outside of a processor's home domain when adaptively solving its subproblem.

Even smoother characteristic functions may be constructed by combining $C^\infty$ exponential bump functions over the coarse mesh triangulation. The bump functions, $w_k$, window the coarse elements, $\tau_k$, and have the form
\begin{align*}
w_k(x,y) = \begin{cases}
e^{\frac{-1}{(\lambda_1\lambda_2\lambda_3)^p}}, \text{ if } \lambda_i > 0,\,\, \forall i \\
0, \text{ otherwise }
\end{cases}
\end{align*}
where the Cartesian coordinates, $(x,y)$, are converted to the Barycentric coordinates, $(\lambda_1,\lambda_2,\lambda_3)$. The value of $p$ controls the steepness of these bump functions, and for the discussion here, $p=1/2$. The windowing functions, $w_k$, need to overlap each other, so the Barycentric coordinates, $(\lambda_1,\lambda_2,\lambda_3)$, are based on the vertices of an extended triangle defined by pushing the vertices of $\tau_k$ away from the triangle's midpoint by a distance of $d_{min}$, where $d_{min}$ is the minimum diameter of an inscribed circle for all elements that are adjacent to $\tau_k$. In this way, the support of the bump function is limited to $\tau_k$ and its immediate neighboring coarse elements. The smooth, overlapping, windowing functions, $w_k$, may now be combined using Shepard's method \cite{Shepard:1968ef} to obtain a PoU over the home domains, $\Omega_l$, with characteristic functions,
\begin{align*}
\Chi_l = \frac{\sum\limits_{\tau_k\in \Omega_l} w_k }{\sum\limits_{\tau_j\in \Omega} w_j}.
\end{align*}
These characteristic functions retain the compact support and $C^\infty$ smoothness of the windowing functions, but now sum to 1 over the domain, generating a PoU. Also note that, while the sum in the denominator appears to include all coarse elements in the domain, only a few of the windowing functions, $w_j$, are nonzero when evaluating $\Chi_l$, namely those that are directly adjacent to $\Omega_l$, resulting in a reasonable number of required operations in order to evaluate $\Chi_l$.

A heuristic comparison between the various PoU's presented here may be made by comparing their ACE refinement patterns. Figure \ref{pou_refinement_patterns:fig} shows example characteristic functions from each PoU along with the corresponding subproblem ACE refinement pattern when applied to a Poisson problem with a smooth right-hand side using 16 processors. The discontinuous PoU does the best job of localizing the refinement in and around the home domain, but there is a lot of refinement near the discontinuities in the right-hand side at corners of home domains. The $C^0$ PoU smooths out this error, resulting in more evenly distributed refinement, but more refinement thus focuses away from the home domain. The $C^\infty$ PoU combines qualities of the other two, both doing a decent job of localizing refinement and avoiding excessive refinement near edges or corners of the support of the characteristic function. The overall effect of the choice of PoU on NIRD performance is discussed further and better quantified in Section \ref{NIRD_numerical_tests:sec}.

\begin{figure}
\includegraphics[width=\textwidth]{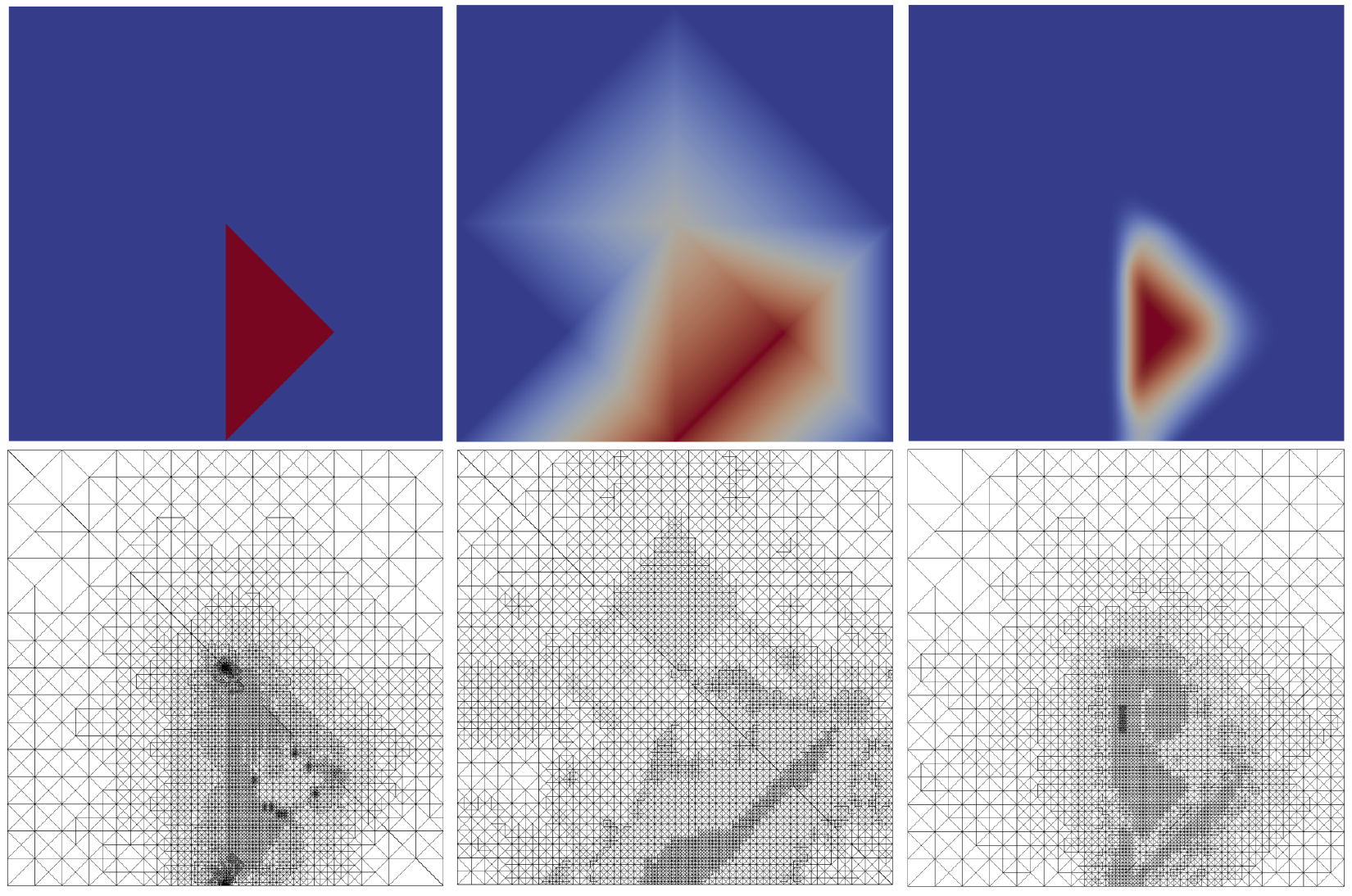}
\caption{Visualization of example characteristic functions (top) and the resulting subproblem refinement patterns (bottom) for the discontinuous (left), $C^0$ (middle), and $C^\infty$ (right) PoU's when applying NIRD to Poisson with a smooth right-hand side using 16 processors.}
\label{pou_refinement_patterns:fig}
\end{figure}


\subsection{Subproblem functional measurement}
\label{NIRD_subproblem_functionals:sec}

Another important note pertaining to the PoU's is that, regardless of the PoU selected, it is possible that $\mathbf{f}_l$ may not be in the range of $L$, even if $\mathbf{f}$ is. That is, applying the characteristic function $\Chi_l$ may push the subproblem right-hand side out of the range of $L$. Thus, $\mathbf{f}_l$ may be decomposed as 
\begin{align*}
\mathbf{f}_l = L\delta\mathbf{u}_l + \boldsymbol{\phi}_l ,
\end{align*}
where $\delta\mathbf{u}_l$ is the continuum solution to the subproblem (i.e. the continuum minimizer for equation (\ref{subproblem_min:eq})) and $\boldsymbol{\phi}_l\in Ker(L^*)$ is nontrivial. This does not change the fact that $\mathbf{u} = \mathbf{u}_0^h + \sum_l \delta\mathbf{u}_l$, as can be seen through a simple inner product argument. Thus, the fact that $\mathbf{f}_l\notin Ran(L)$ does not affect NIRD on a fundamental level: the solution updates found in each subproblem are still correct. This fact can be problematic, however, when measuring the LSF for the subproblems for adaptive refinement. The functional, $||L\delta\mathbf{u}_l^h - \mathbf{f}_l||$, measures the nontrivial component, $\boldsymbol{\phi}_l$, which prevents the functional from converging with refinement. To remedy this, one can subtract off $\boldsymbol{\phi}_l$ when measuring the functional, resulting in the measurement $||L\delta\mathbf{u}_l^h - (\mathbf{f}_l - \boldsymbol{\phi}_l)|| = ||L(\delta\mathbf{u}_l^h - \delta\mathbf{u}_l)||$. This functional now converges with refinement and provides a better indication of where adaptive refinement should occur. 

Calculating the $Ker(L^*)$ components, $\boldsymbol{\phi}_l$, clearly varies from problem to problem, depending on the operator $L$. In some cases, such as a simple div-curl system in which $Ker(L^*)$ is a one-dimensional space, the $\boldsymbol{\phi}_l$ components can be found through a cheap, direct calculation. The problems studied in Section \ref{NIRD_numerical_tests:sec} present a bit more difficulty when calculating the $\boldsymbol{\phi}_l$ components. For these systems, $Ker(L^*)$ is the infinite-dimensional space expressed by
\begin{align*}
Ker(L^*) = \left\{ \begin{bmatrix}
-\nabla^\bot \psi \\ 0 \\ \psi
\end{bmatrix} \right\},
\end{align*}
where $\psi$ is any function in $H^1$. In this case, to find $\boldsymbol{\phi}_l$ for a NIRD subproblem, it is necessary to find $\psi_l$ such that

\begin{align*}
\psi_l &= \argmin_\psi \left|\left|\mathbf{f}_l - \begin{bmatrix} -\nabla^\bot \psi \\ 0 \\ \psi \end{bmatrix}\right|\right|^2. 
\end{align*}
It is not generally possible to directly calculate the solution to this minimization, but the solution may be approximated. Thus, the implementation used for this paper solves the above minimization for $\psi$ in finite element space, $W_l^{(1)}$, a scalar subspace of the full vector finite element space $W_l$ associated with processor $l$'s subproblem. This results in an approximation, $\boldsymbol{\phi}_l^h$, to $\boldsymbol{\phi}_l$, and the error in this approximation should converge at the same rate as the error in the approximation to the subproblem solution, $\delta\mathbf{u}_l$. Thus, the modified functional has the form 
\begin{align*}
||L\delta\mathbf{u}_l^h - (\mathbf{f}_l - \boldsymbol{\phi}_l^h)|| = ||L(\delta\mathbf{u}_l^h - \delta\mathbf{u}_l) + (\boldsymbol{\phi}_l^h-\boldsymbol{\phi}_l)|| \leq ||L(\delta\mathbf{u}_l^h - \delta\mathbf{u}_l)|| + ||\boldsymbol{\phi}_l^h-\boldsymbol{\phi}_l||, 
\end{align*}
and convergence of this functional is again restored, resulting in superior local error measurement over the naive functional, $||L\delta\mathbf{u}_l^h - \mathbf{f}_l||$. Figure \ref{subproblemFunctionalMeasurement:fig} shows convergence of both the naive and modified functionals for each subproblem of NIRD applied to a Poisson problem using 16 processors and linear finite elements. The modified functional achieves the expected $O(h)$ convergence, whereas the naive functional stalls at the level of $||\boldsymbol{\phi}_l||$ for each processor. 

The overall effect on NIRD performance of this modification to the functional measurement on the subproblems actually seems to be somewhat minimal in practice, at least for the problems studied in this paper. The stalling naive functional measurement can lead to somewhat suboptimal ACE refinement patterns on the subproblems, but ultimately the effect is usually not large, and the overall accuracy of NIRD, while improved by modifying the functional, is not drastically changed especially after the first iteration. Thus, in cases where finding the $Ker(L^*)$ components is quite easy (as in the case of the div-curl system above), computing the modified functional is probably worthwhile, whereas in cases where finding these components requires significant computational effort, naive functional measurement is likely better in terms of accuracy per cost. There may exist, however, certain problems for which calculating the $Ker(L^*)$ components is essential to NIRD performance, and searching for such cases is a topic of future research.

\begin{figure}[b!]
\centering
\includegraphics[width=\textwidth]{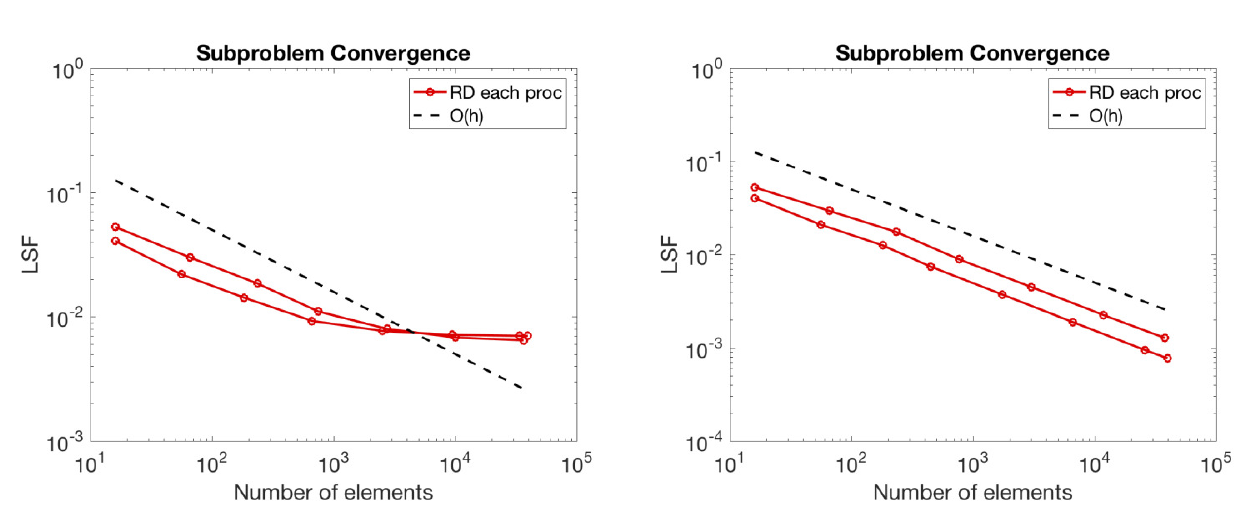}
\caption{Convergence of the LSF for the subproblems on each of 16 processors (note that many of the curves lie on top of each other due to the symmetry of this problem) for Poisson's equation with naive functional measurement (left) and measurement with the $Ker(L^*)$ components removed from the functional (right).}
\label{subproblemFunctionalMeasurement:fig}
\end{figure}




\section{NIRD convergence theory}
\label{NIRD_convergence_theory:sec}

This section describes the most up to date convergence theory for NIRD. This theory utilizes heuristic assumptions that are supported by the numerical results shown in Section \ref{NIRD_numerical_tests:sec}. Developing rigorous convergence theory for NIRD is a challenging task that is still ongoing. While NIRD may appear to share some similarities with existing domain decomposition methods such as additive Schwarz, there are some crucial differences between these methods that severely limit applicability of existing theory. The local problems solved and recombined in traditional additive Schwarz methods are truly local to a subset of the domain (that is they solve only over a patch of the full domain using artificially imposed, local boundary conditions) and have finite overlap with other subproblems. This paradigm forms the basis of the theory developed for Schwarz methods. In NIRD, on the other hand, the ``local" subproblems are, in fact, global in that they solve over the entire computational domain with the original boundary conditions and, thus, have global overlap with all other subproblems. If each processor were able to use the same, globally refined mesh to solve their subproblem, the method would incur no error in that the NIRD solution would be the same as the exact union mesh solution. Thus, combining solutions from subproblem meshes with very different refinement patterns is the source of error between the NIRD solution and the exact union mesh solution, and quantifying this error proves to be difficult.

Section \ref{NIRD_prev_conv:sec} examines key assumptions used in a previous convergence proof from \cite{Appelhans} and measures relevant constants in practice for a simple Poisson problem. This serves to identify which bounds actually hold in practice with nice constants and which do not. Linking numerical results back to the convergence theory in this way has not been done previously for NIRD and is a crucial step in developing the proper heuristic assumptions. Motivated by the results of Section \ref{NIRD_prev_conv:sec}, Section \ref{NIRD_conv:sec} formulates updated assumptions that appear to be realistic in practice and also modifies the structure of the previous convergence proof to account for overlapping PoU's (the proof from \cite{Appelhans} assumes the use of a discontinuous PoU). Thus, the proof presented here reflects the best current understanding of NIRD and yields a clear path forward for further development of the algorithm and the convergence theory.


\subsection{Examining previous assumptions for NIRD convergence}
\label{NIRD_prev_conv:sec}

The previous convergence proof in \cite{Appelhans} suggests that, for many elliptic problems, NIRD requires only a single iteration in order to achieve a solution with functional error within a small factor of that achieved by a traditional nested iteration solve using the same resources. The argument for why such immediate convergence should occur utilizes the heuristic assumptions and important bounds stated below. Recall that $\Omega_k$ refers to processor $k$'s home domain, $\delta\mathbf{u}_k^h$ and $\delta\mathbf{u}_k$ are, respectively, the adaptively refined, discrete approximation and continuum solution for processor $k$'s subproblem, $\mathbf{u}$ is the continuum solution to the original problem, and $\mathbf{u}_1^h$ is the first NIRD iteration approximation to $\mathbf{u}$.

\begin{assumption}
Assume that there exists an $O(1)$ constant, $C_s$, such that
\begin{align*}
|| L(\delta\mathbf{u}_l^h -\delta\mathbf{u}_l) ||_{\Omega_k} &\leq C_{s} || L(\delta\mathbf{u}_k^h -\delta\mathbf{u}_k) ||_{\Omega_l}, &\forall l\neq k.
\end{align*}
\label{symmetry:ass}
\end{assumption}
\begin{assumption}
Assume that there exists an $O(1)$ constant, $C_\rho$, such that
\begin{align*}
\sum_{l\neq k} || L(\delta\mathbf{u}_k^h -\delta\mathbf{u}_k) ||_{\Omega_l} \leq \left(C_{\rho} \sum_{l\neq k} || L(\delta\mathbf{u}_k^h -\delta\mathbf{u}_k) ||_{\Omega_l}^2 \right)^{1/2},  &\forall l.
\end{align*}
\label{exponential:ass}
\end{assumption}

These assumptions are generally motivated by the exponential decay of the Green's function for elliptic PDE's. A more thorough discussion of their motivation may be found in \cite{Appelhans}. The above assumptions are used to establish the following bound that plays a crucial role in the convergence proof:

\begin{align}
||L(\mathbf{u}_1^h - \mathbf{u})||_{\Omega_k} &\leq C_b || L(\delta\mathbf{u}_k^h -\delta\mathbf{u}_k) ||_{\Omega_k},\, \forall k=1...P,
\label{C_b_david:eq}
\end{align}
where
\begin{align}
C_b = \left(1 + C_{s} \left(\frac{C_{\rho}(1-\hat{Q})}{\hat{Q}} \right)^{1/2}\right),
\label{C_b_david_def:eq}
\end{align}
\begin{align}
\hat{Q} = \min_k Q_k\left(\frac{\min\limits_{\tau \subset \Omega_k}|| L(\delta\mathbf{u}_k^h -\delta\mathbf{u}_k) ||_\tau}{\max\limits_{\tau\subset\Omega}|| L(\delta\mathbf{u}_k^h -\delta\mathbf{u}_k) ||_\tau} \right)^2.
\label{Q_hat:eq}
\end{align}
Here, the $\tau$'s are elements in processor $k$'s adaptively refined mesh, and $Q_k$ is the ratio of elements placed in the home turf to the total number of elements after the adaptive subproblem solve on processor $k$. Again, one may refer to \cite{Appelhans} for the derivation of bound (\ref{C_b_david:eq}). 

The assumptions and logic of the proof state that each of the constants, $C_s$, $C_\rho$, $\hat{Q}$, and subsequently $C_b$ should be close to 1 and independent of problem size in order to show that NIRD achieves an accurate solution on the first iteration. In order to investigate the validity of these assumptions and the overall strategy of the proof, it is possible to directly measure these constants in practice. Consider a simple Poisson problem, 
\begin{align*}
-\Delta p &= f , \, \text{ on } \Omega = [0,1]\times[0,1],\\
p &= 0, \, \text{ on } \Gamma ,
\end{align*}
discretized with FOSLS using linear elements. NIRD is known to converge to a solution with accuracy that is very close to that of the union mesh solution in one iteration for this problem (see \cite{Appelhans} or Section \ref{NIRD_numerical_tests:sec} of this paper). 
Evaluating the above line of proof by measuring key constants for this simple test problem shows that the key bound (\ref{C_b_david:eq}) seems to hold in practice while some of the heuristic assumptions and calculation leading to that bound appear either not to hold or to predict values which are far too pessimistic. Table \ref{david_constants:tab} measures the quantities of interest referenced above while scaling the number of processors, $P$, from 64 up to 1,024 and keeping the number of elements in the adaptively refined meshes per processor fixed at 20,000. The first column measures the smallest value, $C_s$, that satisfies Assumption \ref{symmetry:ass}. Note that this quantity grows with problem size, indicating this assumption does not hold in practice, even for a simple Poisson problem. The next column, however, measures a similar quantity,
\begin{align}
\tilde{C}_s = \max_{k} \frac{ \sum\limits_{l\neq k} || L(\delta\mathbf{u}_l^h -\delta\mathbf{u}_l) ||_{\Omega_k} }{ \sum\limits_{l\neq k}|| L(\delta\mathbf{u}_k^h -\delta\mathbf{u}_k) ||_{\Omega_l} } .
\end{align}
Note that, in order for $\tilde{C}_s$ to be an $O(1)$ constant, it is no longer necessary to have some exact symmetry between processor $k$'s error over $\Omega_l$ and processor $l$'s error over $\Omega_k$ for each pair of processors, $k$ and $l$. Rather, the sums above must be of similar size, which relaxes the assumption somewhat. Interestingly, Table \ref{david_constants:tab} shows $\tilde{C}_s$ is significantly smaller than $C_s$ and also remains essentially constant as $P$ grows. Note that the proof in \cite{Appelhans} still follows if the bound in Assumption \ref{symmetry:ass} is simply replaced by a bound of the form
\begin{align*}
\sum_{l\neq k} || L(\delta\mathbf{u}_l^h -\delta\mathbf{u}_l) ||_{\Omega_k} &\leq \tilde{C}_s \sum_{l\neq k} || L(\delta\mathbf{u}_k^h -\delta\mathbf{u}_k) ||_{\Omega_l}.
\end{align*}
Motivating or justifying this bound from a theoretical perspective is not as straightforward as the simpler symmetry assumption used in the original proof, but a bound of this form may, nevertheless, be useful in future development of the convergence theory.

The next quantity measured in Table \ref{david_constants:tab} is the smallest value, $C_\rho$, that satisfies Assumption \ref{exponential:ass}. Again, even for this simple Poisson problem, $C_\rho$ is somewhat large, and may or may not be growing with $P$. Thus, this bound is, at best, quite pessimistic. Similarly, the use of $\hat{Q}$ in the proof is arrived at through the use of simple algebra and requires no further assumptions whatsoever. As shown in Table \ref{david_constants:tab}, however, the value of $\hat{Q}$ is quite small and shrinks as $P$ grows. As a result, it's use in establishing bound (\ref{C_b_david:eq}) is extremely pessimistic to the point where it is no longer useful.

Lastly, Table \ref{david_constants:tab} shows a bottomline comparison between the value for $C_b$ that is predicted by the theory vs. the best value for the bound (\ref{C_b_david:eq}) as measured in practice, that is, 
\begin{align}
C_b &= \left(1 + C_{s} \left(\frac{C_{\rho}(1-\hat{Q})}{\hat{Q}} \right)^{1/2}\right), \\
\tilde{C}_b &=  \max_{k} \frac{||L(\mathbf{u}_1^h - \mathbf{u})||_{\Omega_k}}{|| L(\delta\mathbf{u}_k^h -\delta\mathbf{u}_k) ||_{\Omega_k}}.
\label{C_b_measured:eq}
\end{align}
Clearly, the actual bound, (\ref{C_b_measured:eq}), is satisfied quite nicely, as $\tilde{C}_b$ is small and constant with respect to $P$, whereas the theory stemming from the various heuristic assumptions predict a much larger value, $C_b$, which grows with $P$. This suggests that, while Assumptions \ref{symmetry:ass} and \ref{exponential:ass}, as well as some pessimistic quantities such as (\ref{Q_hat:eq}) do not do a good job of describing the behavior of NIRD in practice, the bound (\ref{C_b_david:eq}) that they aim to establish \emph{does}, in fact, appear to be reasonable. In fact, Section \ref{NIRD_numerical_tests:sec} provides additional numerical evidence to support the hypothesis that NIRD performs well when the bound (\ref{C_b_david:eq}) is satisfied.


\subsection{An updated NIRD convergence proof}
\label{NIRD_conv:sec}

Motivated by the observations of the Section \ref{NIRD_prev_conv:sec}, this section presents a generalization of the previous proof that is applicable to arbitrary PoU's and does away with some of the previous assumptions and bounds. In doing so, this proof preserves and highlights what appear to be crucial assumptions that correlate well with good NIRD performance (as supported by numerical results in Section \ref{NIRD_numerical_tests:sec}) and meaningfully extends the theoretical framework beyond the previous convergence proof, which assumed the use of a discontinuous PoU. Thus, the new, adapted proof is a much more useful tool for understanding the behavior of NIRD and for informing future development of convergence theory.

\begin{assumption}
Assume that the adaptive meshes used here adequately resolve problem features and that error is equidistributed across elements of the adaptive meshes such that it is possible to bound the LSF for some discrete solution, $\mathbf{u}^h$, in terms of the number of elements in the adaptive mesh, $N$. That is, assume
\begin{align*}
||L(\mathbf{u}^h - \mathbf{u})|| \leq C_I N^{-r/d} |\mathbf{u}|_{r+1},
\end{align*}
where $C_I$ is an interpolation constant, $d$ is the dimension of the problem, and $r \leq q$ is the convergence rate determined by the finite element order, $q$, and by the regularity of the problem \cite{Babuvska:1978dd,Zhu:1988dl}. Furthermore, assume bounds of this form are sharp, so that
\begin{align*}
||L(\mathbf{u}^h - \mathbf{u})|| \approx C_I N^{-r/d} |\mathbf{u}|_{r+1}.
\end{align*}
\label{tight_bounds:ass}
\end{assumption}
This assumption simply states that the problem is in the regime where the error is proportional to the number of elements used to resolve the solution. This excludes cases where a problem may be severely under-resolved (in which case some sources of error may not be seen by the under-resolved functional measurement) or where a problem has an exactly recoverable solution (in which case the error goes to zero with a finite number of elements).

\begin{assumption}
Assume that there exists an $O(1)$ constant, $C_0$, such that
\begin{align*}
||L(\mathbf{u}_1^h - \mathbf{u})||_{\tau_k} \leq  C_0 \sum_{l\in \mathcal{L}_k} ||L( \delta\mathbf{u}^h_{l} - \delta\mathbf{u}_{l})||_{\tau_k} ,\,\, \forall \tau_k,
\end{align*}
where the $\tau_k$'s are the coarse elements in the preprocessing mesh, and the set $\mathcal{L}_k$ is defined for each coarse element $\tau_k$ as
\begin{align*}
\mathcal{L}_k = \{l : \text{supp} (\Chi_l) \cap \tau_k \neq \emptyset\}.
\end{align*}
\label{C_0:ass}
\end{assumption}
This assumption is motivated by numerical results and appears to hold in practice when NIRD performs well as shown in Section \ref{NIRD_numerical_tests:sec}. It is also analogous to the bound (\ref{C_b_david:eq}) that is a crucial part of the previous convergence theory. The motivating idea here is that error contributions from processors with a zero right-hand side over a given piece of the domain will be small over that piece of the domain. That is, the primary error in the NIRD solution locally comes from processors that have non-trivial right-hand sides.

\begin{assumption}
Assume a modified interpolation bound of the form
\begin{align*}
||L (\delta\mathbf{u}^h_{l_k} - \delta\mathbf{u}_{l_k})||_{\tau_k} \leq C_I h_k^r  |\delta\mathbf{u}_{l_k}|_{r+1,\tau_k} ,\,\, \forall \tau_k,
\end{align*}
where $l_k = \arg \max_{l\in\mathcal{L}_k} ||L (\delta\mathbf{u}^h_l - \delta\mathbf{u}_l)||_{\tau_k} $, $C_I$ is constant, $h_k$ is the maximum element diameter in processor $l_k$'s subproblem mesh over $\tau_k$, and $r \leq q$ is the convergence rate determined by the finite element order, $q$, and by the regularity of the problem. Furthermore, similar to Assumption \ref{tight_bounds:ass}, assume that the above bound still holds when replacing the step size, $h_k$, by $N_U^{-1/d}$, where $N_U$ is the number of elements in the NIRD union mesh.
\label{interp:ass}
\end{assumption}
A bound of this form is motivated by the fact that a similar bound is guaranteed by finite element theory over the entire domain for each subproblem. The ability to localize such a bound to a subset of the domain cannot generally be guaranteed. This should not be an unreasonable assumption in practice, however, since ACE should approximately equidistribute error among elements in the adaptiveley refined subproblem meshes and the coarse elements, $\tau_k$, where subproblem $l_k$ has nontrivial right-hand side, should contain a significant portion of the total elements used for that subproblem. For the purposes of this proof, assume that the PoU is chosen with sufficient smoothness such that the regularity of the subproblem is not reduced as compared to the original PDE problem and right-hand side. This ensures that $|\delta\mathbf{u}_{l_k}|_{r+1}$ is bounded and that Assumption \ref{interp:ass} is both meaningful and reasonable. It is worth noting, however, that in practice, higher-order convergence has been observed for the subproblems even when the subproblems have reduced regularity due to the use of a discontinuous PoU (see Section \ref{ho_elements:subsec}). This is likely due to the fact that the discontinuities in the right-hand side occur only along element boundaries, thus preserving smoothness on element interiors.

\begin{assumption}
Assume that there exists an $O(1)$ constant, $C_n$, such that
\begin{align}
|\delta\mathbf{u}_{l_k}|_{r+1,\tau_k} \leq \frac{C_n}{|\mathcal{L}_k|} |\delta \mathbf{u}|_{r+1,\tau_k} ,\,\, \forall \tau_k.
\end{align}
where $\delta\mathbf{u}$ is the continuum solution to the global residual equation.
\label{seminorm:ass}
\end{assumption}
A similar assumption was used in \cite{Appelhans}, where the discussion was limited to a discontinuous PoU and $|\mathcal{L}_k| = 1$. The inclusion of $|\mathcal{L}_k|$ in the denominator serves to account for overlapping PoU's. To see the rationale for this, consider the extreme case of globally overlapping characteristic functions, $\Chi_l = 1/P,$ $\forall l$. In this case, the above assumption is trivially satisfied with $C_n = 1$, since $\delta\mathbf{u}_l = (1/P)\delta\mathbf{u},$ $\forall l$ and $|\mathcal{L}_k| = P,$ $\forall k$. 

\begin{theorem}
Under Assumptions \ref{tight_bounds:ass}, \ref{C_0:ass}, \ref{interp:ass}, and \ref{seminorm:ass}, the first NIRD iterate, $\mathbf{u}_1^h$, achieves least-squares functional error within a small, constant factor, $K$, of traditional methods:
\begin{align*}
||L(\mathbf{u}_1^h - \mathbf{u})|| &\leq K ||L(\mathbf{u}_T - \mathbf{u}) ||,
\end{align*}
where $\mathbf{u}_T$ is the solution arrived at using a global nested iteration process using $N_T$ elements in the final, adaptively refined mesh, where $N_T$ is the total number of elements used for all NIRD subproblems.
\label{convergence_proof:thm}
\end{theorem}

\begin{proof}

First, break up the total functional for the solution produced by the first full NIRD iteration, $\mathbf{u}_1^h$, over the $N_c$ coarse elements in the preprocessing mesh, $\tau_k$:
\begin{align*}
||L(\mathbf{u}_1^h - \mathbf{u})||^2_\Omega =  \sum_{k=1}^{N_c} ||L(\mathbf{u}_1^h - \mathbf{u})||^2_{\tau_k}.
\end{align*}
Use the triangle inequality on each term in the sum above:
\begin{align*}
||L(\mathbf{u}_1^h - \mathbf{u})||_{\tau_k} \leq \sum_{l\in\mathcal{L}_k} ||L (\delta\mathbf{u}^h_l - \delta\mathbf{u}_l)||_{\tau_k} + \sum_{m\notin\mathcal{L}_k} ||L (\delta\mathbf{u}^h_m - \delta\mathbf{u}_m)||_{\tau_k}.
\end{align*}
Now, use Assumption \ref{C_0:ass} to sweep the sum over $m\notin \mathcal{L}_k$ into the $O(1)$ constant, $C_0$:
\begin{align}
||L(\mathbf{u}_1^h - \mathbf{u})||_{\tau_k} \leq  C_0 \sum_{l\in\mathcal{L}_k} ||L( \delta\mathbf{u}^h_l - \delta\mathbf{u}_l)||_{\tau_k}.
\label{C0:eq}
\end{align}
Proceed by bounding the sum above by the error contribution from a single processor:
\begin{align}
\sum_{l\in\mathcal{L}_k} ||L (\delta\mathbf{u}^h_l - \delta\mathbf{u}_l)||_{\tau_k} \leq |\mathcal{L}_k| \cdot ||L (\delta\mathbf{u}^h_{l_k} - \delta\mathbf{u}_{l_k})||_{\tau_k},
\end{align}
where $l_k = \arg \max_{l\in\mathcal{L}_k} ||L (\delta\mathbf{u}^h_l - \delta\mathbf{u}_l)||_{\tau_k} $. Combining the above steps yields the following important bound:
\begin{align}
||L(\mathbf{u}_1^h - \mathbf{u})||_{\tau_k} \leq  C_0 |\mathcal{L}_k| \cdot ||L( \delta\mathbf{u}^h_{l_k} - \delta\mathbf{u}_{l_k})||_{\tau_k} ,\,\, \forall \tau_k,
\label{Cb:eq}
\end{align}
where, again, $C_0$ is assumed to be an $O(1)$ constant independent of $P$. Using the above bound in conjunction with Assumptions \ref{interp:ass} and \ref{seminorm:ass} yields,
\begin{align*}
||L(\mathbf{u}_1^h - \mathbf{u})||^2_\Omega &=  \sum_{k=1}^{N_c} ||L(\mathbf{u}_1^h - \mathbf{u})||^2_{\tau_k} \leq   \sum_{k=1}^{N_c} ( C_0 |\mathcal{L}_k| \cdot ||L (\delta\mathbf{u}^h_{l_k} - \delta\mathbf{u}_{l_k})||_{\tau_k}) ^2 \\
&\leq   C_0^2 \sum_{k=1}^{N_c} ( |\mathcal{L}_k| C_I h_k^r  |\delta\mathbf{u}_{l_k}|_{r+1,\tau_k} )^2 \leq  (C_0C_IN_U^{-r/d} )^2 \sum_{k=1}^{N_c}C_n^2  |\delta\mathbf{u}|_{r+1,\tau_k}^2 \\ 
&= (C_0C_IC_n N_U^{-r/d} |\delta\mathbf{u}|_{r+1,\Omega} )^2 .
\end{align*}
Taking the square root yields the bound, 
\begin{align}
||L(\mathbf{u}_1^h - \mathbf{u})||_\Omega \leq C N_U^{-r/d} |\delta\mathbf{u}|_{r+1,\Omega}, 
\end{align}
where $C = C_0C_IC_n$. Directly solving the residual equation on the union mesh yields the same bound as above with a different constant. Thus, assuming such bounds are sharp as stated in Assumption \ref{tight_bounds:ass}, there is a constant, $K_0$, such that
\begin{align}
||L(\mathbf{u}_1^h - \mathbf{u})|| & \leq K_0 ||L(\mathbf{u}_U - \mathbf{u}) ||.
\end{align}
where $\mathbf{u}_U$ is the exact union mesh solution. Using Assumption \ref{tight_bounds:ass}, the accuracy of the exact union mesh solution, $\mathbf{u}_U$, may be related to the accuracy of the traditional nested iteration solution, $\mathbf{u}_T$, by,
\begin{align}
||L(\mathbf{u}_U - \mathbf{u}) || \leq K_1Q^{-r/d} ||L(\mathbf{u}_T - \mathbf{u}) ||
\end{align}
where $Q = N_U/N_T$. It is also important to note here that that the constant $K_1$ also accounts for potentially suboptimal element distribution in the union mesh generated by NIRD. The traditional solution, $\mathbf{u}_T$, is generated using global adaptive mesh refinement on the original problem and is thus assumed to be optimal in the sense of achieving the lowest possible error given $N_T$ available elements. The union mesh, on the other hand, is generated by unioning meshes that are adaptively refined on independent subproblems, so there is no such guarantee of optimality for the union mesh. Numerical results in Section \ref{NIRD_numerical_tests:sec} indicate that the union mesh generated by NIRD is close to optimal in many cases, though some notable exceptions are pointed out. Finally, combining inequalities and letting $K = K_0K_1Q^{-r/d}$ yields the desired result:
\begin{align*}
||L(\mathbf{u}_1^h - \mathbf{u})|| &\leq K ||L(\mathbf{u}_T - \mathbf{u}) ||.
\end{align*}

\end{proof}

Thus, the first iteration of NIRD yields a solution with LSF within a factor $K$ of that achieved by a traditional solve, and it does so with only a single communication phase with $O(\log P)$ cost. The size of the constant $K = K_1K_0Q^{-r/d}$ is determined by the degree to which the assumptions made in the proof hold. The size of the constants $C_0$ and $Q$ both play a significant role in determining $K$ and are also both directly measurable. Section \ref{NIRD_numerical_tests:sec} shows numerical tests in which these quantities are measured and convergence is studied for some different elliptic test problems.

When considering a discontinuous PoU where each processor's home domain consists of a single element from the coarse preprocessing mesh, the above proof reduces to the same argument made in \cite{Appelhans} while replacing the unrealistic arguments leading to bound (\ref{C_b_david:eq}) with Assumption \ref{C_0:ass}. Future work will focus on establishing realistic arguments that lead to Assumption \ref{C_0:ass} based on a priori knowledge of the PDE problem. The generalized framework presented here is designed to accommodate other PoU's. For example, consider the extreme case in which $\Chi_l = 1/P$ over the entire domain for each processor, $l$. Clearly, these global characteristic functions yield an impractical overall algorithm, but they also trivially lead to $||L(\mathbf{u}_1^h - \mathbf{u})|| = ||L(\mathbf{u}_U - \mathbf{u}) || = K_1 P^{r/d} ||L(\mathbf{u}_T - \mathbf{u}) ||$, which provides a good sanity check for the structure of the proof. The $C^0$ and $C^\infty$ PoU's discussed in this paper or other general overlapping PoU's fall somewhere between the discontinuous and global PoU's, but should still be accommodated successfully by this version of the proof.


\section{Modeling communication cost}
\label{NIRD_performance_model:sec}

This section outlines some simple performance models comparing NIRD with a traditional nested iteration (NI) approach. Similar models and descriptions may be found in \cite{Appelhans}, though the updated models presented here take into account the effects of problem difficulty by including details such as coarsening factors, convergence factors, and number of processor neighbors. The following models subsequently predict larger performance gains for NIRD over traditional NI for some more difficult problems.

Only the latency cost, or the number of communications, is considered here. Bandwidth cost may appear to be significant for NIRD since bulk information for large swaths of the domain are communicated on each communication step in the implementation described above, but the algorithm may be modified such that only boundary data is sent and the $L$-harmonic solution over the communicated portion of the domain is recalculated on the receiving processor. This modification requires extra computation, but the bandwidth cost is reduced to the same complexity as a traditional nested iteration approach \cite{Appelhans}. In specific cases of machines with limited bandwidth, this modification may be worth implementing. Previous modeling and benchmarking of algebraic multigrid solvers, however, indicate that latency cost is the primary bottleneck that needs to be overcome on current and future machines \cite{Gahvari:2011dh,DeSterck:2006et}. Thus, this discussion focuses only on examining the number of communications required to solve a PDE problem using NIRD vs. traditional NI.

\subsection{Model descriptions and comparisson}

For the purposes of this discussion, assume that traditional NI begins by refining on a single processor until some maximum number of elements is reached, then distributes this mesh among all available processors and continues to refine, solving on each level using some fixed number of AMG V-cycles, until a desired finest mesh is obtained. In a weak scaling context, the size of the coarse, single-processor mesh is $E$, and the size of the finest mesh is $N = EP$, where $E$ is the maximum number of elements per processor, $P$ is the number of processors, and $E>>P$. Assume that each refinement multiplies the number of degrees of freedom by some constant $c$, and, for simplicity, assume the AMG coarsening factor is similar, so that the number of levels in the multigrid hierarchy for the V-cycles is equal to the number of NI refinements. There is some amount of error reduction, $\beta$, required on each NI level in order to maintain a solution with accuracy on the order of the discretization, and $\beta$ should be fixed independent of the refinement level (this is a main premise of the inductive proof to show $O(n)$ convergence of full multigrid given in \cite{McCormick:2016vp}). Thus, with a V-cycle convergence factor, $\rho$, that is independent of problem size, NI requires $\log_{1/\rho}\beta$ V-cycles per level. Finally, assume that the number of communications on each level of the V-cycle is a constant, $\nu$. The resulting, approximate communication cost for traditional NI, $C_T$, can then be obtained by summing over the NI levels (note presence of the index $i$ in the sum below accounts for the number of levels in the V-cycle):

\begin{align*}
C_T &\approx \sum_{i = \log_c E}^{\log_c (EP)} \nu(\log_{1/\rho}\beta)i \\
&= \nu (\log_{1/\rho}\beta) ( \log_c E \log_c P  + (1/2)(\log_c P)^2 + \log_c E + (1/2)\log_c P).
\end{align*}
Note that, in the regime where $E>P$, it follows that $\log_c E \log_c P > (\log_c P)^2$. While the $O((\log_c P)^2)$ scaling is the major concern, the constant, $\nu (\log_{1/\rho}\beta)$ should not be ignored. This constant can be as small as about 20 for the easiest of problems, such as a 5-point Laplacian problem, but can also grow dramatically for more difficult problems. Larger stencils induced by higher-dimensional problems or higher-order elements as well as the problem of increased operator complexity on the coarse levels of the AMG hierarchies can greatly increase the value of $\nu$ \cite{Gahvari:2011dh,DeSterck:2006et}. Problems which are more difficult for AMG will also yield larger convergence factors, $\rho$, leading to an increase in the $\log_{1/\rho}\beta$ term. Also, the base of the logarithms, $c$, may vary from problem to problem. For problems that demand highly localized adaptive refinement, the total number of NI levels will grow. Also, AMG may not be able to coarsen some difficult problems as aggressively, leading to growth in the number of levels in the multigrid hierarchy. This is all to say that traditional NI exhibits not only poor asymptotic scaling of communication cost, but also suffers larger communication cost depending on the difficulty of the problem.

NIRD, on the other hand, is agnostic to the linear system solver used. Additional work performed in solving the linear system is all done in parallel for traditional NI and, thus, incurs larger communication cost, whereas this extra effort for difficult problems is performed in serial, independently on each processor in NIRD. For any problem, NIRD incurs a communication cost, $C_N$, of 
\begin{align}
C_N = \alpha\log_2 P,
\end{align}
where $\alpha$ is simply the number of NIRD iterations. As shown in Section \ref{NIRD_numerical_tests:sec}, $\alpha=1$ or 2 is often (though not always) sufficient to provide a very accurate result across a variety of problems and discretizations, even when these problems are more difficult to solve using traditional methods. In particular, the anisotropic and jump-coefficient Poisson problems, and Poisson using high-order elements will require more work to solve the linear systems, but are still scalably solvable with 1 or 2 NIRD iterations.

Figure \ref{commCostModel:fig} shows the predicted number of communications, $C_T$ and $C_N$, for a traditional NI solve and NIRD. The cost for NIRD assumes the use of $\alpha = 2$ NIRD iterations. For traditional NI, predicted performance is shown for both an easy and more difficult problem. The easy problem uses $\nu = 20$ communications per level, which is a reasonable estimate for doing two relaxations, prolongation, restriction, and residual calculation with four processor neighbors as one might expect for a 5-point Laplacian problem. The convergence factor for the easy problem is $\rho = 0.1$, the coarsening factor is $c = 4$, and the amount of reduction needed on each level is $\beta = 9$. For the hard problem, $\nu = 40$, as one might expect for a problem with eight processor neighbors on each level, the convergence factor is $\rho = 0.8$, and the coarsening factor is $c = 2$, as one might expect for a problem that demands semicoarsening. Each problem uses $E = 10^6$ elements per processor. As shown in Figure \ref{commCostModel:fig}, even for the easy problem, which is basically an ideal, textbook case for full multigrid, traditional NI requires almost two orders of magnitude more communications than NIRD. For the hard problem, this difference is magnified, with traditional NI requiring between three and four orders of magnitude more communications than NIRD. Thus, there is a clear benefit here for NIRD when attempting to solve more difficult problems. The communication cost of traditional methods will continue to rise as the difficulty of the linear systems that need to be solved increases, whereas NIRD retains its already comparatively small communication cost.

\begin{figure}
\centering
\includegraphics[width=0.5\textwidth]{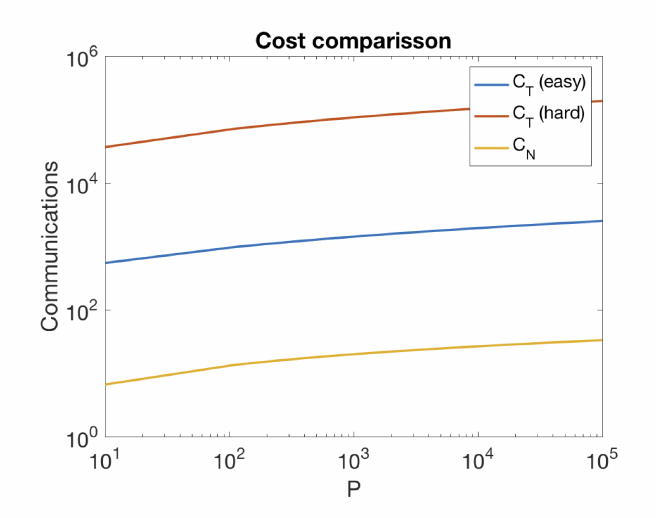}
\caption{Plot of model predictions of number of communications vs. number of processors for NI vs. NIRD.}
\label{commCostModel:fig}
\end{figure}


\section{NIRD numerical results}
\label{NIRD_numerical_tests:sec}

Numerical results from \cite{Appelhans} have already shown the NIRD algorithm to have performance benefits over traditional methods on high-performance machines (where communication is a dominant part of the cost) for simple Poisson and advection-diffusion problems, and the performance model discussed in Section \ref{NIRD_performance_model:sec} predicts even larger performance benefits for more difficult problems. 
The focus here is to provide numerical results that support the heuristic assumptions used in Section \ref{NIRD_conv:sec} and to provide an enhanced understanding of the behavior of NIRD while also applying the algorithm to a wider set of test problems than previously studied and examining performance using different PoU's. The general form of the test problems used in this section is

\begin{align}
-\nabla\cdot(  A \nabla p) + \mathbf{b} \cdot \nabla p &= f , \, \text{ on } \Omega = [0,1]\times[0,1],\\
p &= 0, \, \text{ on } \Gamma_D, \\
\mathbf{n}\cdot \nabla p &= \mathbf{0}, \, \text{ on } \Gamma_N ,
\label{nird_test_problem:eq}
\end{align}
where $\mathbf{b}$ is a vector representing the direction and strength of the advection term, $\mathbf{n}$ is the unit outward normal vector, and 
\begin{align}
A &= \alpha \begin{bmatrix}
1 & 0 \\
0 & \epsilon
\label{diff_coeff_and_anisotropy:eq}
\end{bmatrix}
\end{align}
represents the diffusion coefficient, $\alpha$, along with grid-aligned anisotropy with strength $\epsilon$. 

\subsection{Poisson}

Beginning with the simplest example, consider a Poisson problem, $\alpha = 1, \epsilon = 1, \mathbf{b} = \mathbf{0}$, with zero Dirichlet conditions on the entire boundary. Let the right-hand side be defined by the smooth function,
\begin{align*}
f_0(x,y) &= \sin(\pi x)\sin(\pi y).
\end{align*}
Figure \ref{poisson_f4_scaling_conv:fig} shows NIRD convergence for this Poisson problem with the smooth right-hand side, $f_0$, using a discontinuous PoU and scaling the number of processors, $P$, from 64 to 256 to 1,024 while keeping the number of elements allowed in the finest mesh per processor, $E$, fixed at 20,000. The LSF's shown are normalized by the LSF achieved by the solution on the preprocessing mesh, and the reference lines plotted are the accuracy expected using traditional nested iteration (using ACE refinement) using $N_U$ or $N_T$ elements, where $N_U$ is the number of elements in the union mesh produced by NIRD and $N_T$ is the total number of elements used by NIRD across all subproblems. As expected, NIRD performs well on this problem, producing an accurate solution (that is, a solution that has functional error within a small factor of that achieved by a traditional solution) within the first one or two iterations. Recall that each NIRD iteration requires only $\log_2 P$ communications, so the number of communications required to solve the problem with NIRD is significantly less than what is required by a traditional nested iteration approach. Another thing to note about this problem is that $N_U$ is quite close to $N_T$, meaning that most of the elements used by NIRD end up in the union mesh. Also, the NIRD iterates obtain accuracy close to the traditional solution with $N_U$ elements, indicating that the union mesh produced by NIRD is achieving nearly optimal element distribution, even for the discontinuous PoU. Comparing the performance of NIRD using different PoU's, as shown in Figure \ref{poisson_f4_pou_conv:fig}, shows that the discontinuous and $C^\infty$ PoU's have similar performance, whereas for the $C^0$ PoU, there is some widening of the gap between $N_U$ and $N_T$, meaning that the NIRD union mesh is unable to achieve as good accuracy. This can be explained by the fact that the $C^0$ PoU has the largest support of any of the PoU's, and therefore requires the most refinement outside of the home domains for each subproblem.

\begin{figure}[t]
\includegraphics[width=\textwidth]{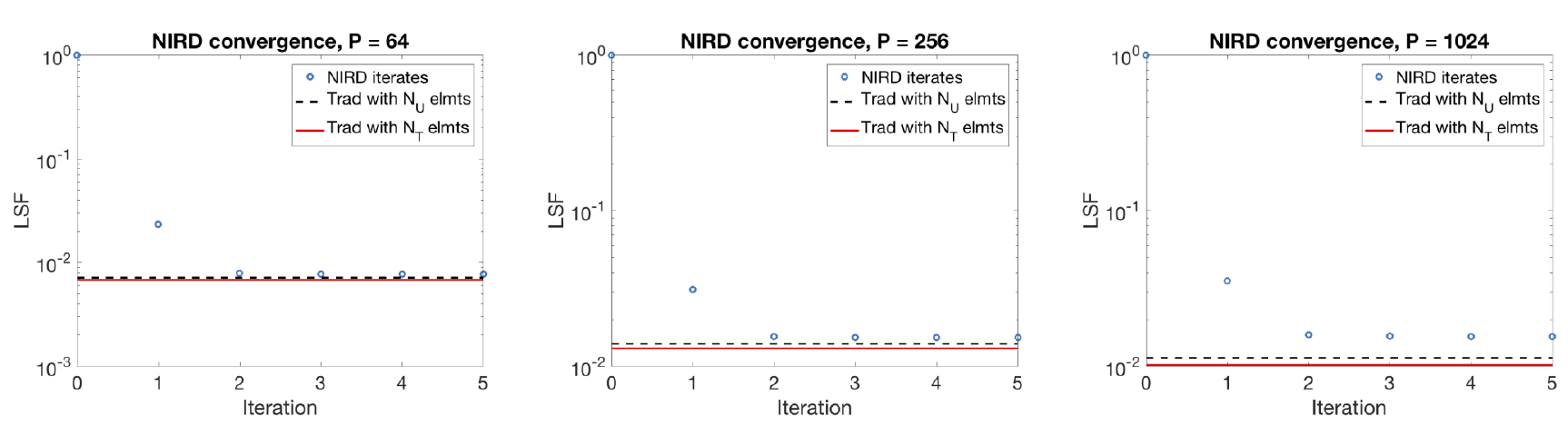}
\caption{NIRD convergence for Poisson with a smooth right-hand side using a discontinuous PoU.}
\label{poisson_f4_scaling_conv:fig}
\end{figure}

\begin{figure}[t]
\includegraphics[width=\textwidth]{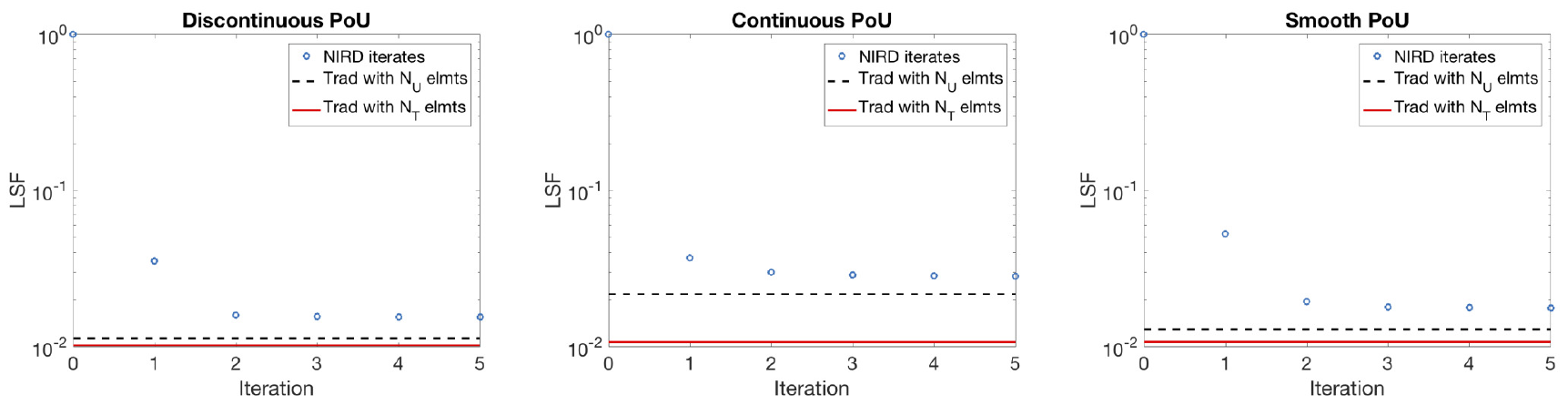}
\caption{Compare NIRD convergence using different PoU's for Poisson with a smooth right-hand side and 1,024 processors.}
\label{poisson_f4_pou_conv:fig}
\end{figure}

Table \ref{poisson_f4:tab} shows some measured quantities for Poisson with a smooth right-hand side. First, $\eta_1/\eta_0$, where $\eta_1$ and $\eta_0$ are the largest and smallest LSF's over home domains after the preprocessing step, and $N_c$, the number of elements in the preprocessing mesh, are shown. This gives an idea of how good a job the preprocessing step does of equidistributing the error among processors. For this simple problem, the error is, unsurprisingly, quite well equidistributed. Next, the smallest constant, $C_0$, that satisfies bound (\ref{C0:eq}) is measured. For this problem, $C_0$ remains small and constant even as $P$ grows. The fact that this is correlated with excellent NIRD convergence on the first iteration lends some validity to the line of proof given in Section \ref{NIRD_conv:sec}. Finally, the table also shows $Q^{(i)}$, the ratio of the number of elements in the union mesh on the $i$'th NIRD iteration vs. the total elements used by NIRD, and $K^{(i)} =||L\mathbf{u}^h_i - \mathbf{f}|| / ||L\mathbf{u}^h_{T} - \mathbf{f}||$, where $\mathbf{u}^h_{T}$ is the solution arrived at by a traditional solve using the same number of elements. One trend present here, which holds for many of the test problems tried throughout this section, is that $Q^{(i)}$ grows on the second iteration (sometimes quite significantly), which is part of the reason why the overall functional for the second NIRD iterate is noticeably smaller than the for the first. The major takeaway here, however, is that even on the first NIRD iteration, the functional error for the NIRD solution is within a factor of about 2 or 3 of the error expected for a traditional solve using the same resources. Thus, NIRD is achieving accuracy which is quite close to traditional methods with minimal communication.

Table \ref{poisson_f0:tab} shows very positive results for NIRD applied to the Poisson problem with the more oscillatory right-hand side,
\begin{align*}
f_1(x,y) &= \alpha(x,y)\sin(3\sqrt{P}\pi x)\sin(3\sqrt{P}\pi y), 
\end{align*}
where $\alpha(x,y)$ yields random values in $[-1,1]$ over the domain such that the amplitudes of the sine humps for $f_1$ vary randomly while retaining the function's continuity. Note that for the smooth right-hand side, $f_0$, the same problem is solved as the number of processors, $P$, grows, but for $f_1$, the problem actually changes with $P$. Also, since $f_1$ is uniformly oscillatory and cannot be adequately resolved to produce useful adaptive mesh refinement by the small number of elements used in the preprocessing step, uniform mesh refinement is used during preprocessing for this problem. This may account for the jump in the quantity $\eta_1/\eta_0$ seen in the table. NIRD performance for the oscillatory right-hand side is even better than for the smooth right-hand side, which is notable since the problem with the more oscillatory right-hand side (and subsequently more oscillatory solution) is traditionally more difficult to solve, requiring finer meshes to properly resolve the solution. For this problem, NIRD converges on the first iteration (as shown by the small values for $K^{(1})$), and the measured value for $C_0$ is small and constant across all PoU's and processor counts. 



\subsection{Advection-diffusion}

The next test problem introduces a constant advection term, $\mathbf{b}$, and modifies the boundary conditions. A Neumann condition is used on the East side of the square domain, and Dirichlet conditions are retained on the North, South, and West sides. The advection term is grid aligned,
\begin{align*}
\mathbf{b} = \begin{bmatrix}
\pm 15\\
0 \end{bmatrix},
\end{align*}
so that the direction of advection is either straight into or away from the Dirichlet condition on the West side, depending on the sign. Note that when using the smooth right-hand side, $f_0$, the solution to the problem has a very different character depending on the sign for $\mathbf{b}$. For a negative sign, the solution has a steep boundary layer against the West side of the domain, whereas this boundary layer is not present in the case where $\mathbf{b}$ is positive. Figure \ref{adv-diff_solns:fig} shows the solutions and corresponding adaptively refined meshes (using global ACE refinement) to the problem for each sign of $\mathbf{b}$. When no boundary layer is present, NIRD converges quickly to an accurate solution, and the value of $C_0$ is small and constant, as shown in Table \ref{adv-diff_no_boundary_layer_f4:tab}. With the boundary layer, however, NIRD performance is clearly not as good, as can be seen in Figure \ref{adv-diff_f4_scaling_conv:fig}. As the number of processor's grows, more NIRD iterations are required to obtain a solution with acceptable accuracy. Notably, Table \ref{adv-diff_f4:tab} also shows significantly larger values for $C_0$ when the boundary layer is present. Again, this supports the hypothesis that Assumption \ref{C_0:ass} is a good indicator of NIRD performance and also suggests that the boundary layer, which is present in the subproblem solutions as well as the global solution, may be causing issues for NIRD by making the NIRD error less ``localized." That is, processors that are far from the West side of the domain still need to resolve the boundary layer there in their subproblem solution, and they do a somewhat poor job of this, resulting in larger values of $C_0$ and poor NIRD convergence on the first iteration. This idea can be more clearly understood by looking at a visualization of an example subproblem solution as shown in Figure \ref{adv-diff_subproblem_plots:fig}. The home domain for the subproblem shown is far from the boundary layer, but the solution retains a steep gradient along the boundary layer and the mesh is heavily refined there. Changing the PoU used by NIRD does not make a substantial difference in performance on the advection-diffusion problem with the boundary layer, as seen in Table \ref{adv-diff_f4:tab}. In each case, $C_0$ is relatively large, and multiple NIRD iterations are required in order to converge to a solution with accuracy comparable to a traditional solve.


\begin{figure}[h]
\includegraphics[width=\textwidth]{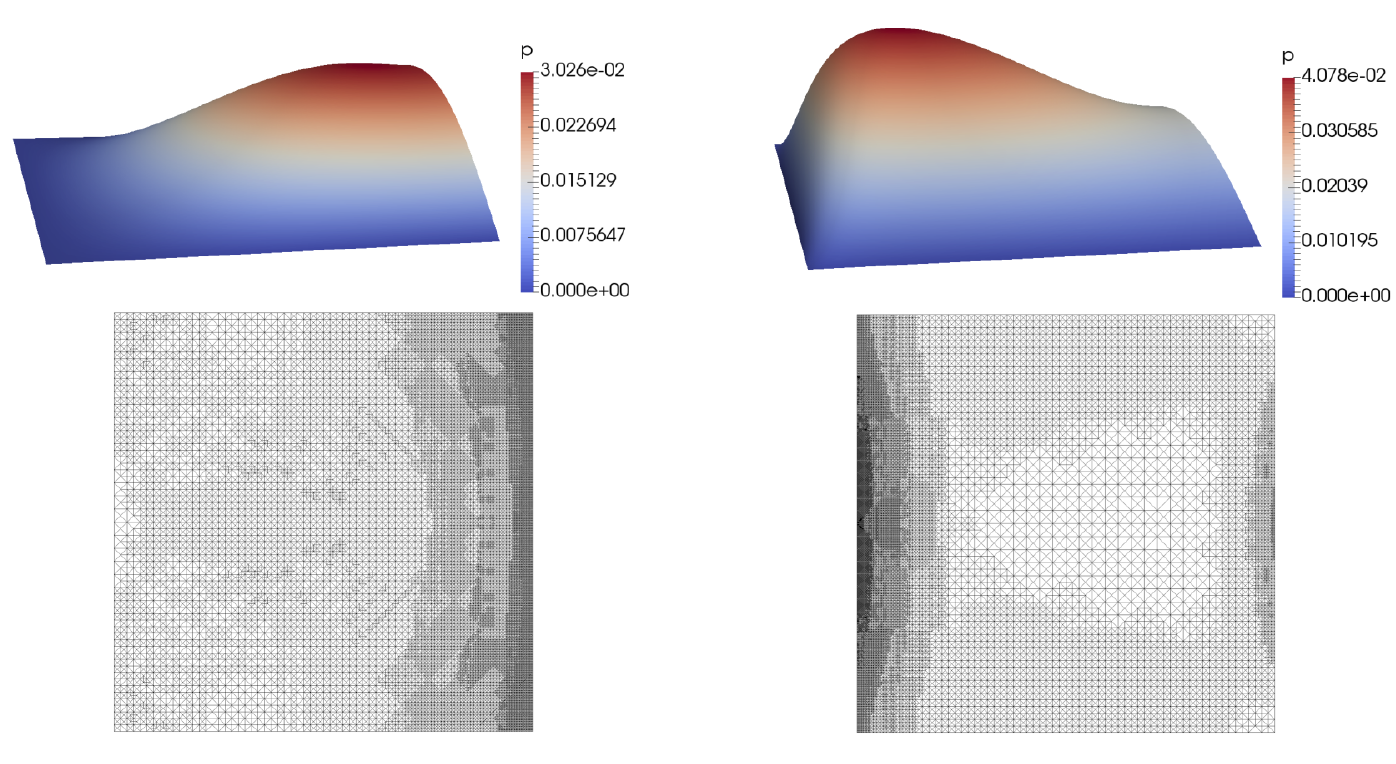}
\caption{Solutions and associated adaptive meshes for the advection-diffusion problem with positive $\mathbf{b}$ (left) and negative $\mathbf{b}$ (right).}
\label{adv-diff_solns:fig}
\end{figure}


\begin{figure}[h]
\includegraphics[width=\textwidth]{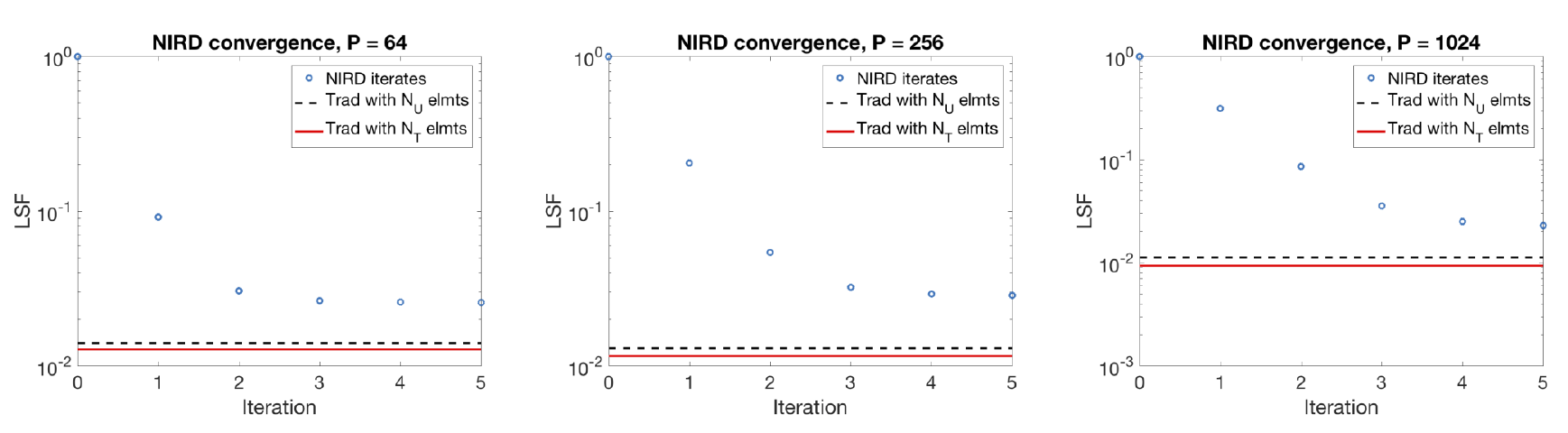}
\caption{NIRD convergence for advection-diffusion with a smooth right-hand side and a boundary layer in the solution using a discontinuous PoU.}
\label{adv-diff_f4_scaling_conv:fig}
\end{figure}


\begin{figure}[h]
\includegraphics[width=\textwidth]{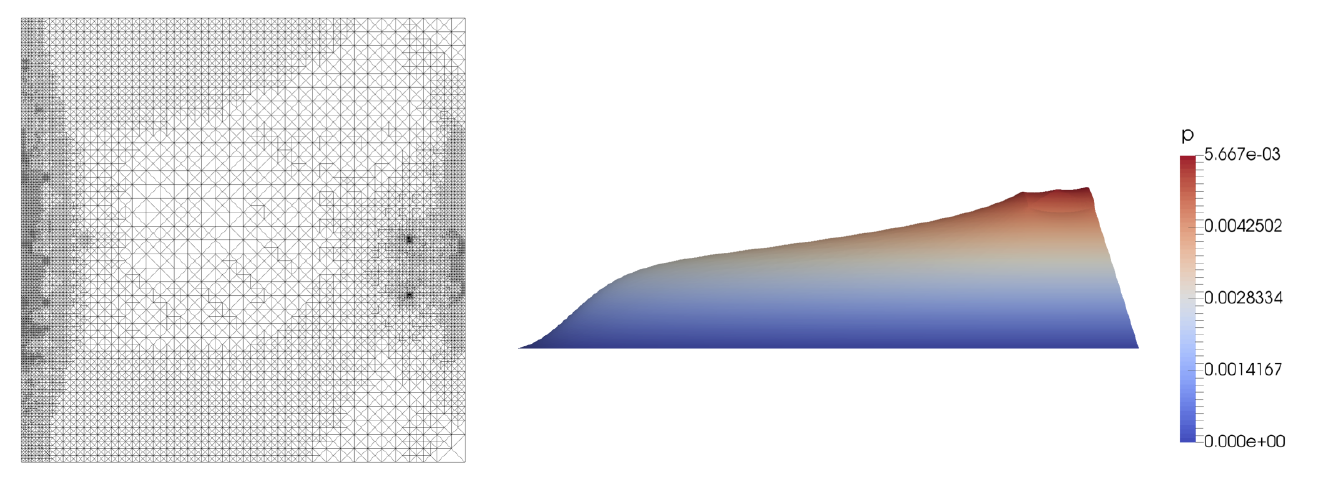}
\caption{An example subproblem mesh (left) and side-view of the solution (right) for a processor far from the boundary layer in the advection-diffusion problem.}
\label{adv-diff_subproblem_plots:fig}
\end{figure}


It is worth noting again that NIRD is actually still performing quite well on the advection-diffusion problem with the boundary layer, even though it is not achieving the same immediate convergence observed for Poisson. An accurate solution with functional error within a small factor of a traditional solution is achieved in 3 or 4 iterations for $P=$ 1,024, each of which use only $\log_2 P$ communications, making this still significantly cheaper than traditional methods.

Examining the advection-diffusion problem with negative $\mathbf{b}$ and the oscillatory right-hand side, $f_1$, yields yet more evidence that the dominance of the boundary layer feature in the solution is what is causing problems for NIRD. The solution for this problem setup still has a boundary layer on the West side of the domain, but the oscillatory right-hand side results in relatively steep changes in the solution that occur over the whole domain. The effect is that the boundary layer plays a much less significant role, and the error in the NIRD solution is once again dominated by the error contributions from local processors as shown by the small values of $C_0$ in Table \ref{adv-diff_f0:tab}. Subsequently, good NIRD convergence on the first iteration is also restored for this problem, as indicated by the small values of $K^{(1)}$ in Table \ref{adv-diff_f0:tab}.


\subsection{Poisson with anisotropy or jump coefficients}

The next set of test problems include anisotropy or jump coefficients. First, consider a jump-coefficient Poisson problem with $\alpha$ from equation (\ref{diff_coeff_and_anisotropy:eq}) set as
\begin{align*}
\alpha(x,y) = \begin{cases} 1,\,\,\,\,\,\text{if } (x,y)\in [0,0.5]\times[0,0.5]\cup[0.5,1]\times[0.5,1] \\ 100,\,\,\,\,\, \text{if }(x,y)\in (0.5,1]\times[0,0.5)\cup[0,0.5)\times(0.5,1] \end{cases},
\end{align*}
which yields a standard, $2\times 2$ checkerboard pattern of discontinuous coefficients, and set $\epsilon = 1$, $\mathbf{b} = \mathbf{0}$, with Dirichlet conditions on the entire boundary. This problem seems to cause no issues for NIRD, and excellent convergence is observed for the smooth right-hand side using the discontinuous PoU as shown in Figure \ref{jump_coeff_f4_scaling_conv:fig}. Similar to the standard Poisson case, there is no real benefit to using either the $C^0$ or $C^\infty$ PoU for this problem, since the discontinuous PoU is sufficient for good convergence, and when solving the problem with the oscillatory right-hand side, NIRD again performs very well.

\begin{figure}[h!]
\includegraphics[width=\textwidth]{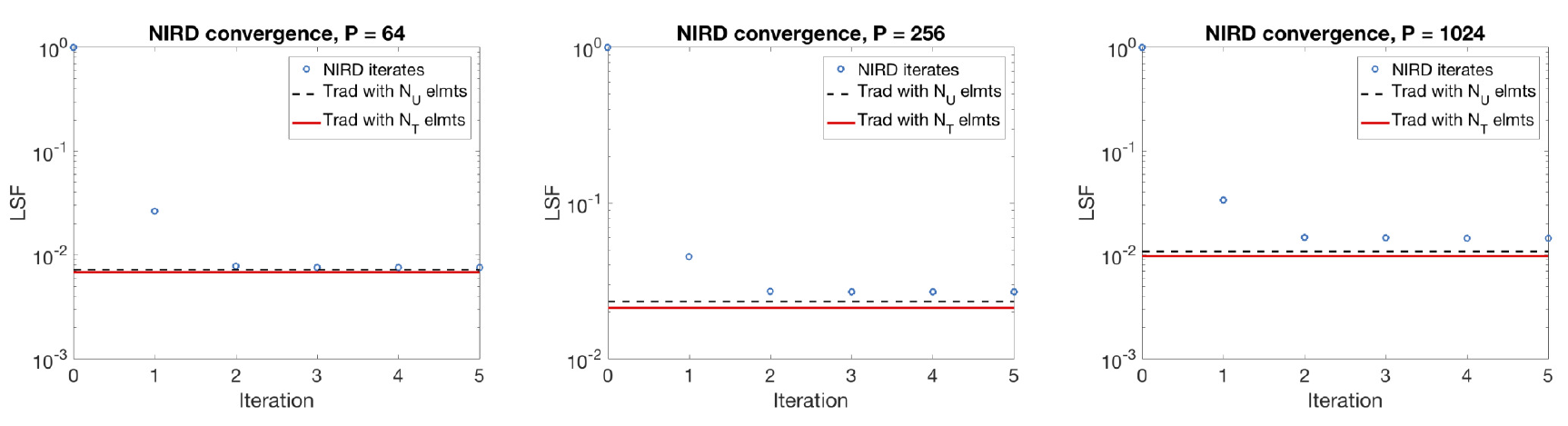}
\caption{NIRD convergence for jump-coefficient Poisson with a smooth right-hand side using a discontinuous PoU.}
\label{jump_coeff_f4_scaling_conv:fig}
\end{figure}

Next, consider the anisotropic Poisson problem where $\epsilon = 0.001$, again keeping all other parameters at their default values, $\alpha = 1$, $\mathbf{b} = \mathbf{0}$, Dirichlet conditions on the entire boundary, and again starting by studying the smooth right-hand side, $f_0$. This problem is of particular interest because it is the only problem presented here for which the smoother PoU's perform better than the discontinuous PoU. As shown in Figure \ref{grid_aligned_anisotropic_f4_pou_conv:fig}, there is a wide gap between the NIRD solution and the optimal solution using $N_U$ elements for the discontinuous PoU, indicating that NIRD is doing a poor job of distributing the elements. When using the discontinuous PoU, NIRD places too many elements near the corners and edges of home domains and, as a consequence, is not globally well refined. The $C^0$ and $C^\infty$ PoU's, however, yield much more sensible refinement patterns, with the $C^\infty$ PoU obtaining more elements in the union mesh. Thus, the $C^0$ and $C^\infty$ PoU's yield scalable NIRD performance with significantly better accuracy on the first and second iterations than the discontinuous PoU, as shown in Table \ref{grid_aligned_anisotropic_f4:tab}. It is also worth noting that the values for $C_0$ are still relatively small and constant across all PoU's, which again correlates to swift NIRD convergence in all cases. The difference in quality of solution here is almost entirely due to the differences in element distribution for the union mesh.

It is again notable that solving the same anisotropic Poisson problem with the oscillatory right-hand side, $f_1$, yields excellent NIRD convergence for all PoU's. As shown by the small values of $K^{(1)}$ in Table \ref{grid_aligned_anisotropic_f0:tab}, NIRD achieves accuracy very close to that achieved by a traditional solve within the first iteration, indicating once more that the oscillatory right-hand side yields a source of error that dominates other problem features (such as those induced by advection or anisotropy) and is also handled well by NIRD. 

\begin{figure}[t]
\includegraphics[width=\textwidth]{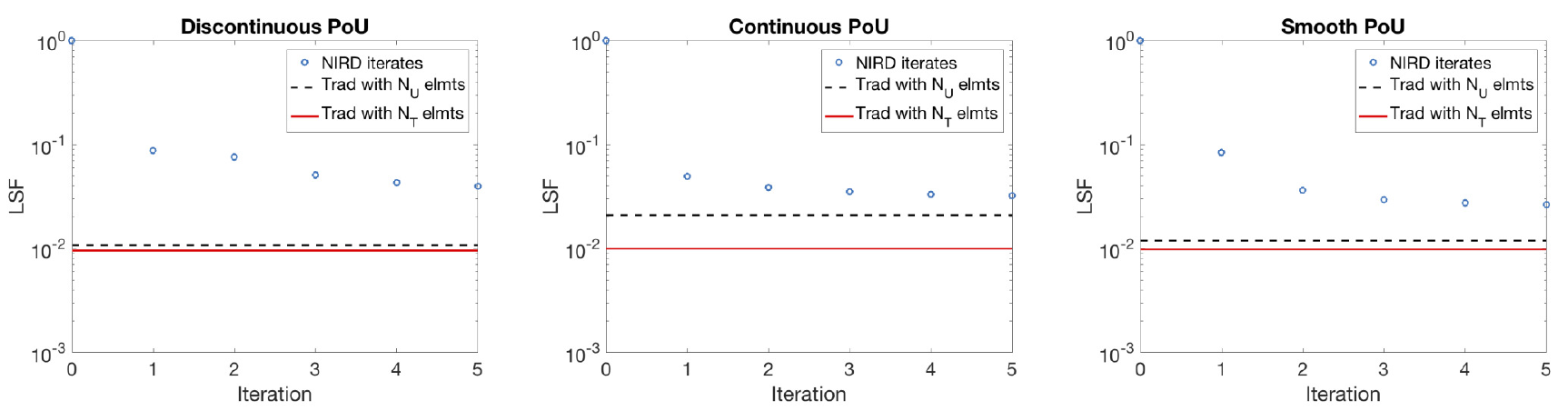}
\caption{Compare NIRD convergence using different PoU's for anisotropic Poisson with a smooth right-hand side and 1,024 processors.}
\label{grid_aligned_anisotropic_f4_pou_conv:fig}
\end{figure}

%



\subsection{Higher-order finite elements}
\label{ho_elements:subsec}

All results from the previous subsections used linear, Lagrange finite elements. This subsection examines the performance of NIRD when applied to higher-order discretizations. In order to justify the use of higher-order elements, consider a solution that has steep gradients. Specifically, this section uses a test problem inspired by Problem 2.9 (the wave front problem) from Mitchell's set of test problems in \cite{Mitchell:2013fl}:
\begin{align*}
-\Delta p &= f_2 , \, \text{ on } \Omega = [0,1]\times[0,1],\\
p &= 0 , \, \text{ on } \Gamma,\\
f_2(x,y) &= -\frac{a + a^3 (r_0^2 - r_1^2)}{r_1 (1 + a^2 (r_0 - r_1)^2)^2}  -\frac{a + a^3 (r_0^2 - r_2^2)}{r_2 (1 + a^2 (r_0 - r_2)^2)^2}, \\
r_0 &= 0.3, \\
r_1(x,y) &= \sqrt{(x-x_1)^2+(y-y_1)^2}, \\
r_2(x,y) &= \sqrt{(x-x_2)^2+(y-y_2)^2}, 
\end{align*}
where $(x_1,y_1) = (0.65,0.65)$, $(x_2,y_2) = (0.35,0.35)$, and $a = 100$. The solution, $p$, to the above problem is shown in Figure \ref{mitchell_29_soln}. This is the test problem used throughout this subsection.

\begin{figure}
\centering
\includegraphics[width=0.7\textwidth]{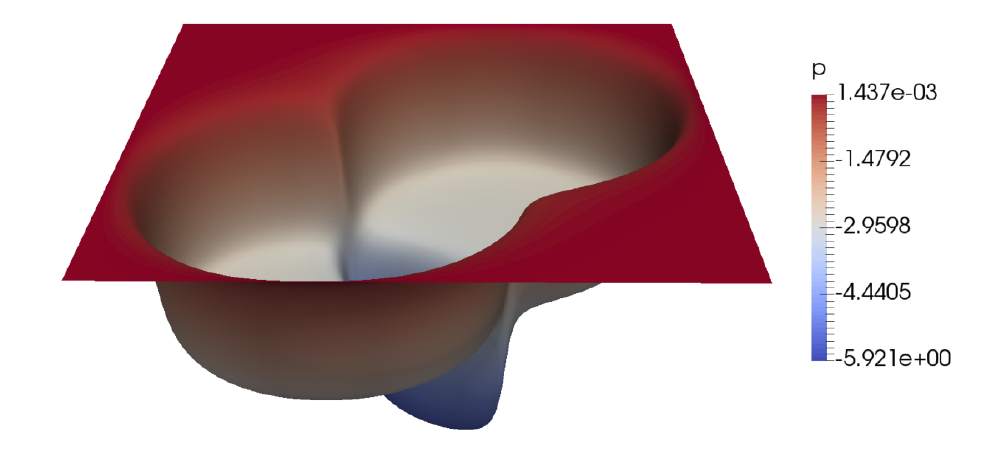}
\caption{Solution to the wave front problem used for testing higher-order elements.}
\label{mitchell_29_soln}
\end{figure}

One important practical note is that increasing element order necessitates decreasing the number of elements per processor, $E$, due to memory constraints. For the tests presented here, $E = 20,000$ is used for degree 1 elements, $E = 5,000$ for degree 2, $E = 2,000$ for degree 3, and $E = 1,000$ for degree 4. Note that this also limits the number of processors that may be used, since NIRD only works well in the regime where $E >> P$. As previously mentioned, this limitation may be overcome by applying NIRD at the level of parallelism between nodes, treating $P$ as the number of nodes, and thus, increasing the available memory and allowing for larger $E$.

Figure \ref{mitchell_degree_performance:fig} shows that NIRD with the discontinuous PoU achieves functional error within a small fraction of that achieved by a traditional solve in one iteration as polynomial order increases from one to four (note the difference in scale for the vertical axis as polynomial order increases). The gap between $N_U$ and $N_T$ does widen somewhat as the degree increases, but this is due to the fact that $E$ is decreasing relative to $P$, leaving less elements to be used for adaptive refinement on the subproblems after the preprocessing step.

\begin{figure}[t]
\centering
\includegraphics[width=0.9\textwidth]{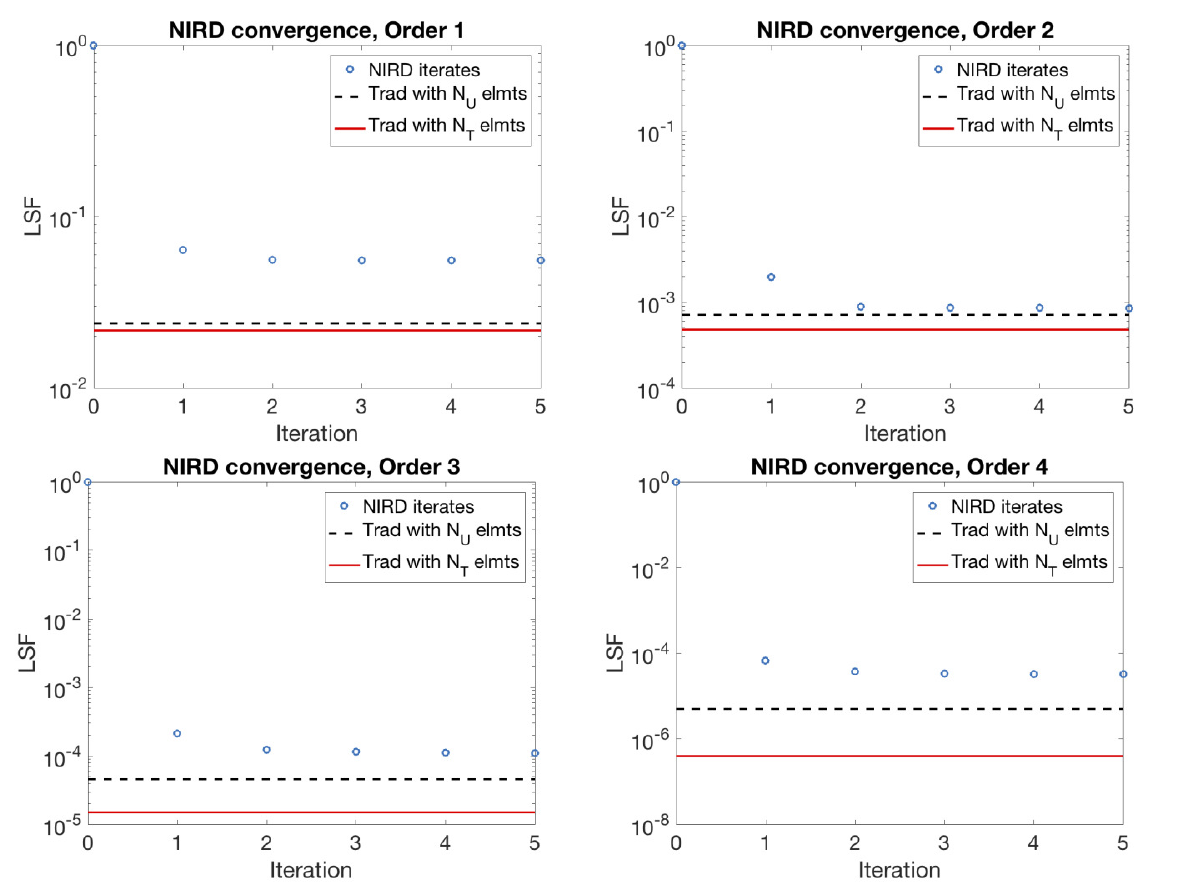}
\caption{NIRD convergence using a discontinuous PoU, 256 processors, and varying the polynomial order of the finite element discretization for the wave front problem (note the difference in scale for the vertical axis as polynomial order increases).}
\label{mitchell_degree_performance:fig}
\end{figure}

When comparing different PoU's, NIRD performance is quite similar across all polynomial orders. The excellent performance achieved by the discontinuous PoU shown in Figure \ref{mitchell_degree_performance:fig} indicates that smoother PoU's are not necessary as polynomial degree of the finite elements increases. This is likely due to the fact that, although the discontinuous PoU clearly reduces the smoothness of the right-hand side, it introduces discontinuities only along element boundaries. For Lagrange elements, the finite element approximation is only $C^0$ at these boundaries anyway for any polynomial order on the interior of the element. Coupling these facts with adaptive mesh refinement yields no degradation in the finite element convergence for the subproblems, even when using the discontinuous PoU. As shown in Figure \ref{mitchell_subproblem_conv:fig}, the LSF for each processor's subproblem (using the discontinuous PoU) converges at a rate of $O(h^q)$, where $q$ is the polynomial degree of the finite elements (here $h$ is defined as $N^{-1/d}$). Note that this also supports the the use of Assumption \ref{interp:ass} in the convergence theory presented in Section \ref{NIRD_conv:sec}.

\begin{figure}[t]
\centering
\includegraphics[width=0.9\textwidth]{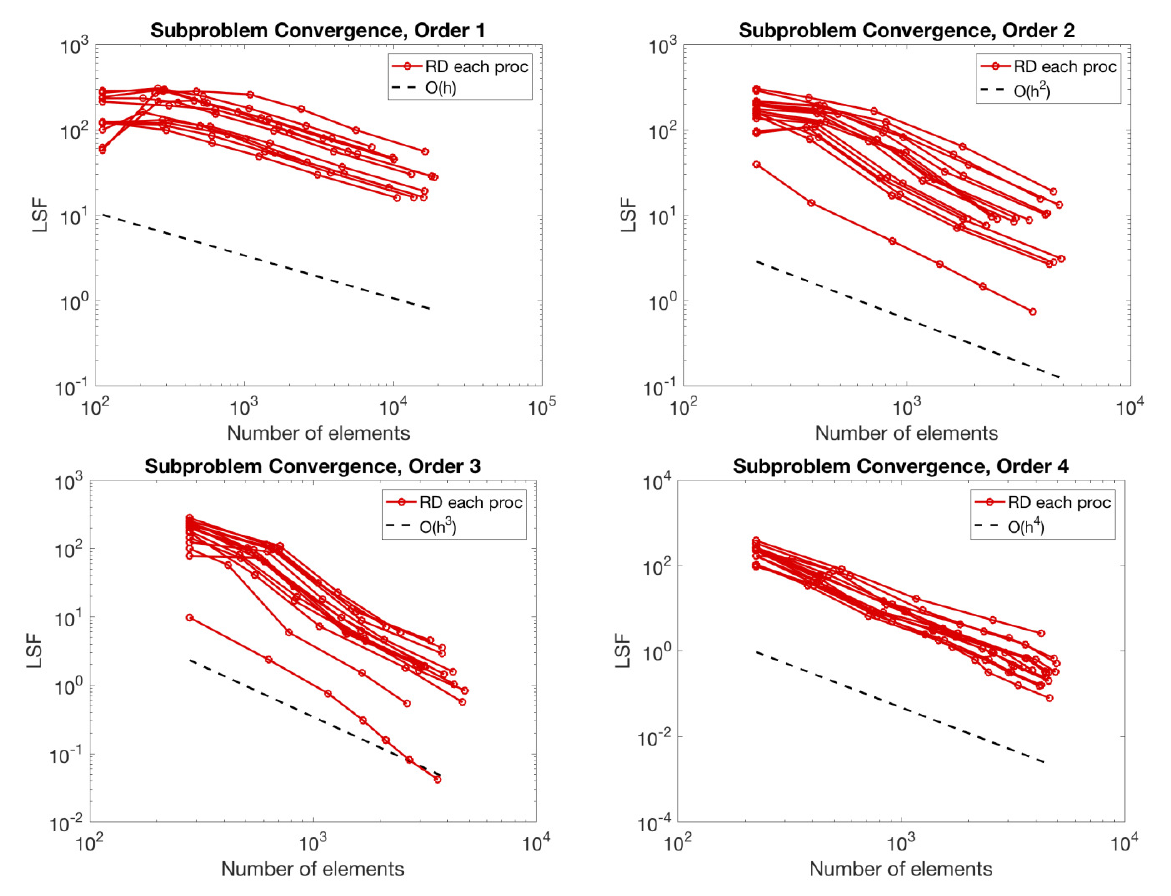}
\caption{NIRD subproblem convergence using a discontinuous PoU, 16 processors, and varying the polynomial order of the finite element discretization for the wave front problem (note the difference in scale for the vertical axis as polynomial order increases).}
\label{mitchell_subproblem_conv:fig}
\end{figure}

\subsection{Discussion}

Throughout studying the above test problems, a few broad observations can be made. First, note that for many of the test problems above, the preprocessing step achieves a fairly small ratio of $\eta_1/\eta_0$ and the NIRD solution is able to recover accuracy fairly close to what traditional ACE refinement using $N_U$ elements achieves. This is especially notable in the cases of the advection-diffusion problem with the boundary layer and the wave front problem, both of which have highly localized features demanding adaptive refinement. This indicates that the adaptive preprocessing strategy developed in Section \ref{NIRD_preprocessing:sec} is performing well and doing a good job of equidistributing error among processors, subsequently allowing NIRD to achieve a union mesh with nearly optimal element distribution.

Second, some general statements comparing the different PoU's tested here may be made. In general, the discontinuous PoU, though the simplest of the PoU's, was sufficient to obtain good NIRD performance on many of the test problems. In fact, even for higher-order finite elements, the NIRD subproblems achieve the proper order of finite element convergence due to the fact that the discontinuities in the right-hand sides are only along element boundaries. With that said, there are problems such as the anisotropic Poisson example presented here for which the $C^0$ and $C^\infty$ PoU's perform better due to the discontinuous PoU's suboptimal element distribution. In general, the discontinuous PoU, which has the smallest support for the characteristic functions, achieves the highest number of elements in the union mesh but has the least optimal element distribution, whereas the $C^0$ PoU, which has the largest support for the characteristic functions, obtains union meshes that are closest to optimal in their element distribution but have significantly fewer elements, and the $C^\infty$ PoU is a middle ground between the other two.

The final important observation is that the numerical results here both validate the line of proof described in Section \ref{NIRD_conv:sec} and contribute additional understanding and intuition about what causes NIRD to work well. The assumption that $C_0$ is small and constant with respect to $P$ is the crux of the updated convergence proof presented in this paper. A strong correlation is observed throughout the numerical results between small values of $C_0$ and good NIRD convergence on the first iteration: for example, $C_0$ is, at most, about 2 or 3 for Poisson, where NIRD converges well on the first iteration, and $C_0$ is as large as 50 for the advection-diffusion problem with a boundary layer, where NIRD converges poorly on the first iteration. This suggests that the line of proof in Section \ref{NIRD_conv:sec} is valid in practice and can be predictive of NIRD performance. This is a step in the right direction for the convergence theory, but a better understanding of how to bound $C_0$ using a priori knowledge of the PDE is still needed. 

A further interesting note is that NIRD performs very well (and $C_0$ is small) across all the test problems when an oscillatory right-hand side is used. The intuitive explanation here is that the use of an oscillatory right-hand side generates NIRD subproblems in which the dominant error is focused inside the support of the characteristic function, where the solution is more difficult to resolve given a more oscillatory right-hand side. This dominance of error in a compactly supported region is one of the guiding principles that enables NIRD to work in the first place, but starting with a highly oscillatory right-hand side seems to magnify this dominance, resulting in exceptional NIRD performance.

\section{Conclusions}

This paper has studied Nested Iteration with Range Decomposition (NIRD), a low-communication solver for elliptic PDE's. The work presented here demonstrates that NIRD is capable of achieving excellent convergence even when faced with a variety of difficulties in the elliptic problem or higher-order finite elements. In fact, the performance model developed here shows that NIRD provides even larger benefits over traditional methods for some more difficult problems. The enhancements and additions to the implementation of NIRD presented in this paper, namely, an adaptive preprocessing step, $C^0$ and $C^\infty$ partitions of unity, and more accurate a posteriori error estimation on the subproblems, further improve NIRD performance on more difficult problems, making the method overall more robust and viable in practice. In addition, the use of numerical results to study assumptions and support the arguments of the convergence theory for NIRD provides valuable insight into the behavior of the method and lays the groundwork for future research. In particular, the results and theory presented here suggest that NIRD performs best when the localized right-hand side is responsible for a dominant amount of the error in the NIRD subproblems as opposed to other global features that may be far from a given processor's home domain. Future research will continue to address this idea by testing NIRD on other problems that may present difficulties through the presence of global features such as reentrant corners.

NIRD is part of an emerging trend towards designing novel algorithms that seek to limit communication rather than optimizing or modifying existing solver techniques. This paper demonstrates NIRD's great promise as a low-communication solver that delivers superior accuracy for the number of communications performed when compared with traditional approaches. NIRD also has some significant advantages over other low-communication methods such as the parallel adaptive mesh algorithm of Bank and Holst \cite{Bank:1999uq, Bank:2006hn} or Mitchell's full domain partition multigrid method \cite{Mitchell:2004hz, Mitchell:2016vg, Mitchell:1998kr}, and thus represents a promising approach in novel, low-communication solvers for elliptic PDE's. Such algorithms are sure to rise in prominence as modern computer architectures continue to grow and the dominant cost continues to shift further away from computation towards communication, especially latency cost. While modification and optimization of existing algorithms can achieve some substantial savings in communication cost, the only way to drastically reduce the number of communications and even change the scaling of communication cost is to develop novel algorithms with this as their premise.

Application areas with a need for solving elliptic PDE's on the largest of modern machines should be implementing and using low-communication algorithms. The NIRD algorithm has now been shown to work well on a wide set of test problems and is also fairly straightforward to implement in the framework of an existing finite element package. Introducing the NIRD algorithm into a larger community will not only provide many users with a much faster method for their elliptic PDE solves but also provide the next testbed for the algorithm, hopefully yielding additional insight into which problems and applications the method is best suited for. 

A major open question of this research is whether algorithms like NIRD may be applied to non-elliptic PDE problems. The central idea of NIRD is to decompose the PDE problem rather than the computational domain. There may be other decompositions of the PDE (besides decomposing the right-hand side) that would work well for non-elliptic PDE's, and exploration of this topic will be a focus of future research.


\bibliographystyle{unsrt}
\bibliography{bibfile}


\begin{table}[p!]
\centering
\begin{tabular}{c | c | c | c | c | c | c}
$P$ & $C_s$ & $\tilde{C}_s$ & $C_\rho$ & $\hat{Q}$ & $C_b$ & $\tilde{C}_b$ \\
\hline
64   & 6.89   & 1.41 & 29.90 & 1.83E-04 & 2.78E+03 & 1.82 \\
256  & 49.01  & 1.93 & 64.75 & 1.39E-06 & 1.18E+05 & 1.52 \\
1024 & 258.10 & 2.05 & 73.41 & 2.87E-07 & 2.25E+06 & 2.09
\end{tabular}
\caption{Measured constants for Poisson with a smooth RHS.}
\label{david_constants:tab}
\end{table}

\begin{table}[h!]
\centering
\begin{tabular}{ c | c || c | c | c | c | c | c | c }
PoU & $P$ & $\eta_1/\eta_0$ & $N_c$ & $C_0$ & $Q^{(1)}$ & $K^{(1)}$ & $Q^{(2)}$ & $K^{(2)}$ \\
\hline
& 64      & 2.68 & 64   & 1.82 & 0.34 & 3.35 & 0.92 & 1.16      \\
Discts. & 256  & 2.81 & 952  & 1.71 & 0.55 & 2.37 & 0.89 & 1.19 \\
        & 1024 & 6.17 & 2524 & 2.29 & 0.42 & 3.21 & 0.83 & 1.56 \\
\hline
        & 64   & 2.68 & 64   & 0.82 & 0.15 & 3.62 & 0.19 & 2.75 \\
$C^0$    & 256  & 2.81 & 952  & 0.97 & 0.35 & 2.49 & 0.41 & 2.00 \\
        & 1024 & 6.17 & 2524 & 0.89 & 0.18 & 3.44 & 0.24 & 2.80 \\
\hline
& 64     & 2.68 & 64   & 0.52 & 0.28 & 3.01 & 0.73 & 1.52       \\
$C^\infty$ & 256  & 2.81 & 952  & 0.85 & 0.51 & 2.27 & 0.80 & 1.26 \\
       & 1024 & 6.17 & 2524 & 2.71 & 0.38 & 4.82 & 0.69 & 1.81
\end{tabular}
\caption{Table of relevant values for NIRD using different PoU's for Poisson with a smooth RHS.}
\label{poisson_f4:tab}
\end{table}

\begin{table}[h!]
\centering
\begin{tabular}{ c | c || c | c | c | c | c | c | c }
PoU & $P$ & $\eta_1/\eta_0$ & $N_c$ & $C_0$ & $Q^{(1)}$ & $K^{(1)}$ & $Q^{(2)}$ & $K^{(2)}$ \\
\hline
&  64      & 6.23 & 256   & 1.74 & 0.69 & 1.34 & 0.93 & 1.03   \\
Discts. & 256  & 3.01  & 4096 & 3.44 & 0.56 & 1.50 & 0.74 & 1.14 \\
        & 1024 & 31.36 & 4096 & 2.38 & 0.56 & 1.62 & 0.72 & 1.17 \\
\hline
        & 64   & 6.23  & 256  & 0.91 & 0.48 & 1.79 & 0.51 & 1.49 \\
$C^0$    & 256  & 3.01  & 4096 & 1.56 & 0.46 & 1.69 & 0.51 & 1.43 \\
        & 1024 & 31.36 & 4096 & 1.03 & 0.35 & 2.23 & 0.37 & 1.85 \\
\hline
& 64     & 6.23 & 256   & 0.92 & 0.69 & 1.34 & 0.88 & 1.06       \\
$C^\infty$ & 256  & 3.01  & 4096 & 1.33 & 0.55 & 1.52 & 0.70 & 1.17 \\
       & 1024 & 31.36 & 4096 & 0.93 & 0.55 & 1.63 & 0.67 & 1.23
\end{tabular}
\caption{Table of relevant values for NIRD using different PoU's for Poisson with an oscillatory RHS.}
\label{poisson_f0:tab}
\end{table}

\begin{table}[h]
\centering
\begin{tabular}{ c | c || c | c | c | c | c | c | c }
PoU & $P$ & $\eta_1/\eta_0$ & $N_c$ & $C_0$ & $Q^{(1)}$ & $K^{(1)}$ & $Q^{(2)}$ & $K^{(2)}$ \\
\hline
& 64      & 5.43 & 64   & 2.09 & 0.33 & 2.92 & 0.92 & 1.06       \\
Discts. & 256  & 9.98 & 436  & 2.17 & 0.39 & 2.53 & 0.91 & 1.13 \\
        & 1024 & 9.71 & 1600 & 2.73 & 0.35 & 2.57 & 0.85 & 1.17 \\
\hline
        & 64   & 5.43 & 64   & 1.11 & 0.13 & 3.06 & 0.21 & 2.16 \\
$C^0$    & 256  & 9.98 & 436  & 1.13 & 0.15 & 2.75 & 0.26 & 1.98 \\
        & 1024 & 9.71 & 1600 & 0.98 & 0.12 & 2.94 & 0.22 & 2.23 \\
\hline
& 64     & 5.43 & 64   & 0.44 & 0.20 & 2.98 & 0.73 & 1.26     \\
$C^\infty$ & 256  & 9.98 & 436  & 1.90 & 0.32 & 3.99 & 0.81 & 1.42 \\
       & 1024 & 9.71 & 1600 & 1.87 & 0.31 & 4.23 & 0.75 & 1.48
\end{tabular}
\caption{Table of relevant values for NIRD using different PoU's for advection-diffusion with a smooth RHS and no boundary layer in the solution.}
\label{adv-diff_no_boundary_layer_f4:tab}
\end{table}

\begin{table}[h]
\centering
\begin{tabular}{ c | c || c | c | c | c | c | c | c }
PoU & $P$ & $\eta_1/\eta_0$ & $N_c$ & $C_0$ & $Q^{(1)}$ & $K^{(1)}$ & $Q^{(2)}$ & $K^{(2)}$ \\
\hline
& 64      & 4.95 & 226  & 45.65 & 0.15  & 7.00 & 0.84  & 2.39    \\
Discts. & 256  & 4.70 & 758   & 51.53 & 0.15 & 17.14 & 0.79 & 4.58  \\
        & 1024 & 8.77 & 2062  & 48.57 & 0.15 & 31.96 & 0.69 & 8.76  \\
\hline
        & 64   & 4.95 & 226   & 13.11 & 0.08 & 6.31  & 0.37 & 2.84  \\
$C^0$    & 256  & 4.70 & 758   & 21.43 & 0.05 & 13.78 & 0.29 & 4.24  \\
        & 1024 & 8.77 & 2062  & 22.33 & 0.06 & 28.76 & 0.22 & 7.07  \\
\hline
& 64     & 4.95 & 226  & 11.87 & 0.17  & 7.11 & 0.78  & 2.43        \\
$C^\infty$ & 256  & 4.70 & 758   & 16.39 & 0.18 & 16.63 & 0.71 & 4.28  \\
       & 1024 & 8.77 & 2062  & 19.84 & 0.20 & 34.50 & 0.61 & 10.08
\end{tabular}
\caption{Table of relevant values for NIRD using different PoU's for advection-diffusion with a smooth RHS and a boundary layer in the solution.}
\label{adv-diff_f4:tab}
\end{table}

\begin{table}[h!]
\centering
\begin{tabular}{ c | c || c | c | c | c | c | c | c }
PoU & $P$ & $\eta_1/\eta_0$ & $N_c$ & $C_0$ & $Q^{(1)}$ & $K^{(1)}$ & $Q^{(2)}$ & $K^{(2)}$ \\
\hline
& 64      & 7.33 & 256   & 2.05 & 0.65 & 1.41 & 0.93 & 1.04      \\
Discts. & 256  & 7.04  & 1024 & 2.21 & 0.63 & 1.47 & 0.88 & 1.08 \\
        & 1024 & 29.76 & 4096 & 2.49 & 0.56 & 1.63 & 0.72 & 1.18 \\
\hline
        & 64   & 7.33  & 256  & 1.20 & 0.46 & 1.87 & 0.51 & 1.51 \\
$C^0$    & 256  & 7.04  & 1024 & 1.29 & 0.42 & 1.96 & 0.44 & 1.60 \\
        & 1024 & 29.76 & 4096 & 1.09 & 0.35 & 2.25 & 0.37 & 1.86 \\
\hline
& 64     & 7.33 & 256   & 1.13 & 0.65 & 1.42 & 0.88 & 1.09      \\
$C^\infty$ & 256  & 7.04  & 1024 & 1.12 & 0.61 & 1.46 & 0.83 & 1.11 \\
       & 1024 & 29.76 & 4096 & 1.04 & 0.54 & 1.64 & 0.67 & 1.24
\end{tabular}
\caption{Table of relevant values for NIRD using different PoU's for advection-diffusion with an oscillatory RHS and a boundary layer in the solution.}
\label{adv-diff_f0:tab}
\end{table}

\begin{table}[h]
\centering
\begin{tabular}{ c | c || c | c | c | c | c | c | c }
PoU & $P$ & $\eta_1/\eta_0$ & $N_c$ & $C_0$ & $Q^{(1)}$ & $K^{(1)}$ & $Q^{(2)}$ & $K^{(2)}$ \\
\hline
& 64      & 2.04 & 64   & 1.35 & 0.35 & 14.68 & 0.81  & 10.94       \\
Discts. & 256  & 2.59 & 256  & 1.75 & 0.34  & 14.30 & 0.80  & 11.90 \\
        & 1024 & 7.02 & 2196 & 1.44 & 0.33  & 8.61  & 0.79  & 7.84  \\
\hline
        & 64   & 2.04 & 64   & 0.35 & 0.11  & 4.65  & 0.19  & 3.28  \\
$C^0$    & 256  & 2.59 & 256  & 0.72 & 0.08  & 5.44  & 0.18  & 3.68  \\
        & 1024 & 7.02 & 2196 & 1.05 & 0.14  & 4.97  & 0.24  & 3.80  \\
\hline
        & 64   & 2.04 & 64   & 0.73 & 0.17  & 7.31  & 0.49  & 3.61  \\
$C^\infty$  & 256  & 2.59 & 256  & 0.78 & 0.18  & 7.74  & 0.50  & 3.22  \\
        & 1024 & 7.02 & 2196 & 1.28 & 0.25  & 6.58  & 0.56  & 3.40 
\end{tabular}
\caption{Table of relevant values for NIRD using different PoU's for anisotropic Poisson with a smooth RHS.}
\label{grid_aligned_anisotropic_f4:tab}
\end{table}

\begin{table}[h]
\centering
\begin{tabular}{ c | c || c | c | c | c | c | c | c }
PoU & $P$ & $\eta_1/\eta_0$ & $N_c$ & $C_0$ & $Q^{(1)}$ & $K^{(1)}$ & $Q^{(2)}$ & $K^{(2)}$ \\
\hline
& 64      & 8.69 & 256   & 4.24 & 0.54 & 1.81 & 0.85 & 1.41       \\
Discts. & 256  & 3.33  & 4096 & 6.64 & 0.46 & 1.99 & 0.61 & 1.65 \\
        & 1024 & 14.16 & 4096 & 4.12 & 0.37 & 2.00 & 0.61 & 1.80 \\
\hline
        & 64   & 8.69  & 256  & 1.61 & 0.44 & 1.90 & 0.51 & 2.02 \\
$C^0$    & 256  & 3.33  & 4096 & 3.36 & 0.40 & 1.97 & 0.46 & 1.78 \\
        & 1024 & 14.16 & 4096 & 1.87 & 0.28 & 2.09 & 0.35 & 2.51 \\
\hline
& 64     & 8.69 & 256   & 1.86 & 0.55 & 1.80 & 0.80 & 1.44       \\
$C^\infty$ & 256  & 3.33  & 4096 & 3.31 & 0.45 & 1.95 & 0.58 & 1.59 \\
       & 1024 & 14.16 & 4096 & 1.61 & 0.36 & 1.97 & 0.59 & 1.80
\end{tabular}
\caption{Table of relevant values for NIRD using different PoU's for anisotropic Poisson with an oscillatory RHS.}
\label{grid_aligned_anisotropic_f0:tab}
\end{table}


\end{document}